\documentclass[12pt,a4paper]{amsart}

\usepackage[T1]{fontenc}
\usepackage{amssymb,amsmath,latexsym,amsthm,paralist,a4,mathrsfs,dsfont}
\usepackage{tikz-cd}
\usetikzlibrary{decorations.markings,decorations.pathmorphing}

\makeatletter
\tikzcdset{
  open/.code     = {\tikzcdset{hook, circled};},
  closed/.code   = {\tikzcdset{hook, slashed};},
  open'/.code    = {\tikzcdset{hook', circled};},
  closed'/.code  = {\tikzcdset{hook', slashed};},
 circled/.code  = {\tikzcdset{markwith = {\draw (0,0) circle (.375ex);}};},
 slashed/.code  = {\tikzcdset{markwith = {\draw[-] (-.4ex,-.4ex) -- (.4ex,.4ex);}};},
  markwith/.code ={
    \pgfutil@ifundefined%
    {tikz@library@decorations.markings@loaded}%
    {\pgfutil@packageerror{tikz-cd}{You need to say %
      \string\usetikzlibrary{decorations.markings} to use arrows with markings}{}}{}%
    \pgfkeysalso{/tikz/postaction = {
      /tikz/decorate,
      /tikz/decoration={markings, mark = at position 0.5 with {#1}}}
    }
  },
}
\makeatother

\usepackage[headings]{fullpage}
\usetikzlibrary{babel}
\usetikzlibrary{positioning,calc}
\usepackage{varioref}
\usepackage[pdfpagelabels, pdftex]{hyperref}
\hypersetup{
  pdftitle={Adic spaces},
  pdfauthor={Katharina H\"{u}bner},
  pdfsubject={},
  pdfkeywords={},
  colorlinks=true,    
  linkcolor=red,     
  citecolor=blue,     
  filecolor=blue,      
  urlcolor=blue,       
  breaklinks=true,
  bookmarksopen=true,
  bookmarksnumbered=true,
  pdfpagemode=UseOutlines,
  plainpages=false,
  unicode=true
  }
\usepackage{stmaryrd}
\usepackage[capitalise]{cleveref}
\usepackage[inline]{showlabels}
\usepackage{changes}  
\usepackage{bbold} 
\newcommand{\id}{\mathrm{id}}

\newcommand{\Sp}{\mathrm{Sp}}

\renewcommand{\a}{\mathfrak{a}}

\newcommand{\m}{\mathfrak{m}}
\newcommand{\p}{\mathfrak{p}}

\newcommand{\Z}{\mathds Z}
\newcommand{\N}{\mathds N}
\newcommand{\Q}{\mathds Q}

\newcommand{\F}{\mathds F}
\newcommand{\R}{\mathds R}

\newcommand{\NN}{\mathds{N}}
\newcommand{\DD}{\mathds{D}}

\newcommand{\CC}{\mathds{C}}
\newcommand{\SSS}{\mathds{S}}

\newcommand{\Frac}{\mathrm{Frac}}

\newcommand{\A}{{\mathds A}}
\newcommand{\ZZ}{\mathds{Z}}
\renewcommand{\P}{{\mathds P}}

\newcommand{\disc}{\mathit{disc}}
\newcommand{\QQ}{\mathds{Q}}

\newcommand{\cI}{{\mathscr I}}

\newcommand{\cO}{{\mathscr O}}

\newcommand{\cU}{{\mathscr U}}
\newcommand{\cV}{{\mathscr V}}

\newcommand{\cX}{{\mathscr X}}
\newcommand{\cY}{{\mathscr Y}}

\newcommand{\liso}{\mathrel{\hbox{$\longrightarrow$} \kern-2.4ex\lower-1ex\hbox{$\scriptstyle\sim$}\kern1.7ex}}

\newcommand{\an}{\mathrm{an}}
\newcommand{\val}{\mathrm{val}}
\newcommand{\spm}{\mathrm{sp}}

\newcommand{\Dcirc}{\overset{\circ}{\DD}}

\makeatletter
\newcommand{\zerounderset}[3][\mathord]{%
  #1{\vtop{
    \let\\\cr
    \baselineskip\z@skip\lineskip.25ex
    \ialign{\hidewidth$##$\hidewidth\crcr
      \omit$#3$\cr
      #2\crcr
    }%
  }}%
}
\makeatother

\newtheoremstyle{alexthm}
  {}
  {}
  {\sl }
  {}
  {\bf}
  {.}
  {.5em}
  {}
\theoremstyle{alexthm}

\newtheorem{theorem}{Theorem}[section]
\newtheorem*{theorem*}{Theorem}
\newtheorem{corollary}[theorem]{Corollary}
\newtheorem{proposition}[theorem]{Proposition}
\newtheorem{lemma}[theorem]{Lemma}
\newtheorem*{lemma*}{Lemma}

\newtheorem{exercise}[theorem]{Exercise}

\newtheoremstyle{alexdef}
  {}
  {}
  {\rm }
  {}
  {\bf}
  {.}
  {.5em}
  {}
\theoremstyle{alexdef}
\newtheorem*{example*}{Example}
\newtheorem{example}[theorem]{Example}

\newtheorem{remark}[theorem]{Remark}

\newtheorem{definition}[theorem]{Definition}

\DeclareMathOperator{\Spec}{\mathrm{Spec}}

\DeclareMathOperator{\Spa}{\mathrm{Spa}}
\DeclareMathOperator{\Spf}{\mathrm{Spf}}
\DeclareMathOperator{\Cont}{\mathrm{Cont}}
\DeclareMathOperator{\supp}{\mathrm{supp}}

\DeclareMathOperator*{\colim}{colim}

\DeclareMathOperator{\ad}{\mathrm{ad}}

\DeclareMathOperator{\rk}{\mathrm{rk}}

\definecolor{darklimegreen}{RGB}{31,142,8}

\begin{document}

\hfuzz=4pt
\title{Adic spaces}
\author{Katharina H\"{u}bner}
\email{huebner@math.uni-frankfurt.de}
\date{\today}
\address{Robert Mayer Stra{\ss}e 6-8, 60325 Frankfurt}

\begin{abstract}
 Building up on the theory of Huber pairs presented in John Bergdall's lecture (\cite{bergdall2024huber}) we explain the construction of adic spaces.
 We study some important classes of adic spaces such as rigid analytic spaces and formal schemes and show the connections between them.
 In the course of the lecture we will illustrate the respective concepts with the fundamental examples of the open and closed disc and the affine line.
\end{abstract}

\maketitle
\tableofcontents

\section{Introduction}

These lecture notes are based on a series of lectures at the Spring school ``Non-archime\-dean geometry and Eigenvarieties'' in March~2023 in Heidelberg.
The objective of the first three courses was to give an introduction to the theory of adic spaces.
Like schemes they are locally ringed spaces but with additional structure.
Their structure sheaf carries a topology and every point comes with a valuation on its residue field.
One could say that the foundation of scheme theory is commutative algebra, i.e., the theory of rings and modules.
In this spirit the foundation of the theory of adic spaces would be ``topological'' commutative algebra, i.e. the theory of topological rings and modules.
In order to construct adic spaces we only need to consider topological rings of a certain kind, so called Huber rings.
Moreover, in order to encode the valuations on residue fields we need to study subrings of integral elements of a Huber ring.
A Huber ring~$A$ together with a subring~$A^+$ of integral elements is called a Huber pair and denoted $(A,A^+)$.
An example of a Huber ring would be~$\Q_p$ and an example of a ring of integral elements therein would be~$\Z_p$.
So $(\Q_p,\Z_p)$ is a Huber pair.

This foundation on topological rings has been laid in the first lecture by John Bergdall.
In this course we build on top of that and start with the construction of the adic spectrum of a Huber pair together with its structure sheaf.
This defines affinoid adic spaces.
We can then glue affinoid adic spaces together to obtain adic spaces, our main objects of study.
The category of adic spaces encompasses many different types of spaces.
In fact, rigid analytic varieties, formal schemes, and algebraic varieties can all be found as full subcategories in the category of adic spaces.
We will throw a spotlight at each of these incarnations of adic spaces and show connections between them.
In the process we will get to know many of the fundamental examples of adic spaces such as the closed and open unit disc and the affine line.

\section{The adic spectrum of a Huber pair}

For a Huber pair $(A,A^+)$ we define its adic spectrum.
As a set it is defined as
\[
 \Spa(A,A^+) = \{v:A \to \Gamma \cup \{0\}~\text{continuous valuation} \mid v(A^+) \le 1\}/\sim,
\]
where "$\sim$" refers to  equivalence of valuations.
In other words
\[
 \Spa(A,A^+) = \{v \in \Cont(A) \mid v(A^+) \le 1\}.
\]
If $A^+ = A^\circ$, we often just write $\Spa(A)$ instead of $\Spa(A,A^\circ)$.

For a point $x \in \Spa(A,A^+)$ we use the following notation for the corresponding valuation:
\begin{align*}
 A	 \longrightarrow	& ~\Gamma_x \cup \{0\}	\\
 a	 \longmapsto		& ~|a(x)|.
\end{align*}
This notation suggests that we should think of taking the valuation of~$a$ corresponding to a point~$x$ as evualuating~$a$ at~$x$ in the same way a function is evaluated at a point.

The \emph{support} of a valuation $v$ on a ring~$A$ is the prime ideal
\[
	\supp v := \{a \in A \mid v(a) = 0\}
\]
In fact a valuation~$v$ on~$A$ can equivalently be viewed as a valuation of the field $k(\supp v)$.
We just need to convince ourselves that~$v$ naturally defines a valuation~$\bar{v}$ of $A/\supp v$.
Since~$\bar{v}$ takes the value zero only if the argument is zero, we can extend $\bar{v}$ to the residue field $k(\supp v) = \Frac(A/\supp v)$ by setting
\[
 \bar{v}(\frac{a}{b}) = \frac{\bar v(a)}{\bar v(b)}.
\]
Conversely, any valuation~$\bar{v}$ of a residue field $k(\p)$ for $\p \in \Spec A$ gives rise to a valuation on~$A$ by composing with the natural map:
\[
 v:A \longrightarrow k(\p) \overset{\bar{v}}{\longrightarrow} \Gamma.
\]
This valuation has support $\p$ and the two constructions are inverse to each other.

\begin{example}
 \begin{itemize}
 	\item	The space $\Spa(\QQ_p) = \Spa(\QQ_p,\ZZ_p)$, has only one point, the $p$-adic valuation of $\QQ_p$.
			The trivial valuation is not continuous.
	\item	In $\Spa(\ZZ_p) = \Spa(\ZZ_p,\ZZ_p)$ there are two points, the $p$-adic valuation of $\QQ_p$ and the trivial valuation of $\F_p = \ZZ_p/p\ZZ_p$.
	\item	$\Spa(\QQ,\ZZ)$ has one point for each prime $p$ representing the $p$-adic valuation and an additional point for the trivial valuation.
	\item	Let $(k,k^\circ)$ be an algebraically closed non-archimedean field, e.g. $k = \CC_p$.
			The structure of the open unit disc over~$k$
			\[
			 \DD_k = \Spa(k \langle T \rangle) = \Spa(k\langle T\rangle,k^\circ\langle T\rangle)
			\]
			is already quite complicated.
			The easiest type of points is of the following form:
			Take $\alpha \in k$, with norm $|\alpha| \le 1$ and $r \in (0,1]$.
			This defines a point $x_{\alpha,r} \in \DD_{k}$  given by
			\[
			 |f(x_{\alpha,r})| = \max_i |a_i| r^i
			\]
			for $f = \sum a_i (T-\alpha)^i$.
			The full classification of points of $\DD_k$ can be found in \cite{bergdall2024huber}, \S 1.4.4.
 \end{itemize}
\end{example}

Let us study in some more detail the adic spectrum of an \emph{affinoid field}.
By this we mean a Huber pair $(k,k^+)$, where $k$ is a field with valuation ring~$k^+$ and such that~$k$ is endowed either with the valuation topology or the discrete topology.
In case $k$ carries the valuation topology and~$k^+ \ne k$, we need the extra condition that the valuation be \emph{microbial}, i.e., that~$k^+$ contains a nonzero topologically nilpotent element (see \cite{Hu96}, Definition~1.1.4).
But we are on the safe side if we assume that the rank of~$k^+$ is finite (which is fine for all our examples).
If the topology of~$k$ is discrete, we call the affinoid field $(k,k^+)$ discrete and otherwise analytic.
The standard example of an analytic affinoid field is $(k,k^\circ)$ for a nonarchimedean field~$k$.

Remember that the valuation topology associated with the valuation ring~$k^+$ is the topology defined by the metric\footnote{This is not a metric in the classical sense as it might not take values in~$\R$ (in case the valuation has higher rank) but otherwise satisfies all properties of a metric}
\[
 d(x,y) := |x-y|,
\]
where $|\cdot|$ denotes the valuation associated with~$k^+$.
So the sets
\[
 \{ x \in k \mid |x| < \gamma \}
\]
for varying $\gamma \in |k^\times|$ form a fundamental system of open neighborhoods of $0 \in k$.

\begin{exercise} \label{topology_valuation_ring}
 The sets
 \[
  \{ x \in k \mid |x| \le \gamma\}
 \]
 for $\gamma \in |k^\times|$ are open and also form a fundamental system of open neighborhoods of $0$.
\end{exercise}

At this point we can explain why we need the valuation ring~$k^+$ to be microbial in the analytic case.
Suppose this is the case and let~$\varpi$ be a topologically nilpotent element.

\begin{exercise} \label{topology_affinoid_field}
 The valuation topology on~$k$ coincides with the topology defined by the pair of definition $(k^\circ,(\varpi))$.
 Moreover $k^+$ is open in~$k$ and contained in~$k^\circ$ (which implies that we can alternatively take the pair of definition $(k^+,(\varpi))$).
 In particular, $(k,k^+)$ is indeed a Huber pair.
\end{exercise}

\begin{lemma} \label{k_circ_rank1}
 The subring $k^\circ \subseteq k$ of powerbounded elements is a valuation ring of rank~$1$ with maximal ideal $k^{\circ \circ}$.
\end{lemma}

\begin{proof}
In general, every subring of~$k$ containing a valuation ring of~$k$ (in our case $k^+$) is again a valuation ring.
Here is the argument:
Suppose $x \in k$ is not contained in~$k^\circ$.
Then it is not contained in $k^+$ either.
Since $k^+$ is a valuation ring, $1/x$ is contained in $k^+$, hence in $k^\circ$.

In order to see that $k^{\circ \circ}$ is the maximal ideal of~$k^\circ$, we have to check that every element in $k^\circ \setminus k^{\circ \circ}$ is a unit of $k^\circ$.
If $x \in k^\circ$ is not topologically nilpotent, there is an open neighborhood $U \subseteq k$ of zero not containing any power of~$x$.
Unravelling the definition of the valuation topology we can equivalently formulate this in terms of the valuation $|\cdot|$ of~$k^+$:
There is $\gamma \in |k^\times|$ such that all powers of $x$ are bounded below by $\gamma$:
\[
 |x^n| \ge \gamma \quad \forall n \in \N.
\]
This implies
\[
 \left|\left(\frac{1}{x}\right)^n\right| \le \frac{1}{\gamma} \quad \forall n \in \N,
\]
so the powers of $1/x$ are bounded by $1/\gamma$.
This translates to $1/x$ being powerbounded in the valuation topology.
We conclude that $1/x$ is contained in~$k^\circ$, hence $x$ is invertible.

It remains to convince ourselves that $k^+$ has rank~$1$, or equivalently Krull dimension~$1$.
We claim that $(0) \ne k^{\circ \circ}$ and these are the only prime ideals of $k^\circ$.
The inequality $(0) \ne k^{\circ \circ}$ follows from the existence of the topologically nilpotent unit $\varpi$.
Being a unit just means $\varpi \ne 0$, so $\varpi \in k^{\circ \circ} \setminus \{0\}$.
We already know that $k^\circ$ is a valuation ring (hence in particular local) with maximal ideal $k^{\circ \circ}$.
Any other prime ideal needs to be contained in $k^{\circ \circ}$.
Therefore it suffices to show that the radical ideal generated by any nonzero element~$x$ of $k^{\circ \circ}$ equals $k^{\circ \circ}$.
To see this we take any other element $y \in k^{\circ \circ}$ and use its topological nilpotence to find $n \in \N$ such that
\[
 |y^n| \le |x|
\]
This implies that $a:= y^n/x$ is contained in $k^+$, hence in~$k^\circ$.
In other words, $y$ is contained in the radical ideal generated by~$x$.
\end{proof}

\begin{remark}
 The valuation $|\cdot|^\circ$ corresponding to $k^\circ$ can be obtained from $|\cdot|$ in the following way.
 Remember that an element $\gamma$ of a totally ordered group~$\Gamma$ is \emph{cofinal} if for any other element $\delta \in \Gamma$ there is $n \in \N$ such that
 \[
  \gamma^n \le \delta.
 \]
 Note that $x \in k^\times$ is topologically nilpotent if and only if $|x|$ is cofinal in $|k^\times|$.
 Now we convince ourselves that the subset
 \[
  \Delta:=\{\gamma \in |k^\times| \mid \text{neither $\gamma$ nor $1/\gamma$ is cofinal in $|k^\times|$}\}
 \]
 of $|k^\times|$ is a convex subgroup.
 Then $\Gamma:= |k^\times|/\Delta$ is a totally ordered group and we claim that $|\cdot|^\circ$ is given by the composite
 \[
  k \overset{|\cdot|}{\longrightarrow} |k^\times| \cup \{0\} \twoheadrightarrow \Gamma \cup \{0\}.
 \]
 Notice that $x \in k^\times$ is powerbounded if and only if $1/x$ is \emph{not} topologically nilpotent.
 In this way we identify $k^\circ$ with the set of topologically nilpotent elements $k^{\circ \circ}$ together with those elements $x \in k^\times$ such that neither $x$ nor $1/x$ is topologically nilpotent.
 This is precisely the valuation ring corresponding to the above defined valuation.
\end{remark}

Affinoid fields are important as their adic spectrum is the prototype of something like a point in an adic space.
This is not a hundred percent precise since $\Spa(k,k^+)$ may contain several points as the following example shows.

\begin{example}[The adic spectrum of an affinoid field] \label{spectrum_affinoid_field}
 We want to understand the adic spectrum of an affinoid field $(k,k^+)$.
 We first take a look at all continuous valuations.
 In case the topology of~$k$ is discrete, there are no restrictions and any valuation ring of~$k$ defines a continuous valuation.
 For example the continuos valuations of $\QQ$ (with the discrete topology) are precisely all $p$-adic valuations for different prime numbers~$p$ together with the trivial valuation (this follows from Ostrowski's theorem).
 
 Let us now assume that the valuation ring~$k^+$ is non-trivial, i.e. not equal to~$k$ and $k$ carries the valuation topology.
 In other words $(k,k^+)$ is analytic.
 It is not hard to see that the valuation corresponding to a valuation ring~$R$ is continuous if and only if~$R$ is contained in the subring of powerbounded elements $k^\circ \subseteq k$.
 In our setting, where the topology on~$k$ is the valuation topology defined by the valuation ring~$k^+$, $k^\circ$ is a valuation ring of rank~$1$ (see Lemma~\ref{k_circ_rank1}).
 Having an inclusion of valuation rings
 \[
  R \subseteq k^\circ \subseteq k
 \]
 automatically implies that~$k^\circ$ is the localization of~$R$ at a prime ideal (in our case a prime ideal of height~$1$ since~$k^\circ$ has rank~$1$).
 On the one hand, the image of~$R$ under
 \[
  \pi:k^\circ \twoheadrightarrow k^\circ/k^{\circ \circ}
 \]
 is a valuation ring~$\bar{R}$ and $R = \pi^{-1}(\bar{R})$.
 On the other hand the preimage of any valuation ring of $k^\circ/k^{\circ \circ}$ under~$\pi$ is a valuation ring corresponding to a continuous valuation.
 In this way we see that the continuous valuations of~$k$ correspond to the valuations of $k^\circ/k^{\circ \circ}$.
 Via this correspondence $k^\circ$ corresponds to the trivial valuation on $k^\circ/k^{\circ \circ}$.
 
 Remember that there are no non-trivial valuations on a finite field.
 Therefore, if $k^\circ/k^{\circ \circ}$ is an algebraic extension of a finite field, there is only one continuous valuation of~$k$, namely the one corresponding to~$k^\circ$.
 This is the case for $k = \QQ_p$ or $k=\CC_p$.
 
 Finally let us describe the adic spectrum $\Spa(k,k^+)$.
 By definition it consists of all (equivalence classes of) continuous valuations whose valuation ring contains~$k^+$.
 These correspond to valuation rings $R$ in between $k^+$ and $k^\circ$:
 \[
  k^+ \subseteq R \subseteq k^\circ.
 \]
 All valuation rings of this kind are totally ordered by inclusion and can be obtained by localizing~$k^+$ at a prime ideal.
 In this case the rank of~$R$ equals the height of the prime ideal.
 If the topology of~$k$ is discrete, $k^\circ = k$ and any prime ideal of~$k^+$ is allowed.
 If $(k,k^+)$ is analytic, $k^\circ$ is of rank~$1$ and thus only the prime ideals of height greater or equal to~$1$ occur.
 We conclude that $\Spa(k,k^+)$ has $\rk k^+$ many points if $(k,k^+)$ is analytic and $\rk k^+ + 1$ many points if $(k,k^+)$ is discrete.
\end{example}

For abelian topological groups and commutative topological rings there is a general theory of completion (see \cite{BourbakiTG}, III \S 3).
In our situation the completion $\hat{A}$ of a Huber ring~$A$ is obtained by $I$-adically completing the ring of definition $A_0$ (where $I \subseteq A_0$ is an ideal of definition) and then taking the tensor product:
\[
 \hat{A} = \hat{A}_0 \otimes_{A_0} A.
\]
We obtain a ring of integral elements $\hat{A}^+ \subseteq \hat{A}$ by taking the topological closure of~$A^+$ in~$\hat{A}$.
Then $\hat{A}^+$ is automatically integrally closed in $\hat{A}$ as can be seen by approximation.

\begin{lemma} \label{spectrum_completion}
 Let $(A,A^+)$ be a Huber pair with completion $(\hat{A},\hat{A}^+)$.
 Then we have a natural identification
 \[
  \Spa(\hat{A},\hat{A}^+) \cong \Spa(A,A^+).
 \]
\end{lemma}

\begin{proof}
 The crucial point is that the adic spectrum only contains \emph{continuous} valuations.
 A continuous valuation
 \[
  v : A \to \Gamma
 \]
naturally extends to the completion of~$A$.
For a sequence $(a_n)_{n \in \N}$ in $A$ converging to $a \in \hat{A}$ the continuity of the valuation~$v$ implies that for any $\gamma \in \Gamma$ the valuation of the difference $v(a_n - a_m)$ is less than $\gamma$ for $m$ and $n$ big enough.
If the sequence $(v(a_n))_{n \in \N}$ becomes arbitrarily small for $n \to \infty$, we set $v(a) = 0$.

Otherwise we claim that $(v(a_n))_{n \in \N}$ stabilizes.
By assumption there is $C \in \Gamma$ with $C > 0$ such that
\[
 v(a_n) \ge C
\]
for $n$ big enough.
Let $\gamma \in \Gamma$ with $0 < \gamma < C$.
Then for $n$ and $m$ big enough, $v(a_n - a_m) < \gamma$ and therefore
\[
 v(a_m) = v(a_n + (a_m - a_n)) = v(a_n).
\]
Here we used that
\[
 v(a_m - a_n) < \gamma < C \le a_n
\]
such that we even get an equality in the strong triangle inequality.
Hence we can set $v(a)$ equal to the constant value $v(a_n)$ for $n$ big enough.

Of course one needs to check that the construction does not depend on the chosen sequence $(a_n)_{n \in \N}$ and indeed defines a continuous valuation on $\hat{A}$.
\end{proof}

We equip $\Spa(A,A^+)$ with the subspace topology from $\Cont(A)$.
This is the topology generated by the subsets of $\Spa(A,A^+)$ of the form
\[
 R\left(\frac{f_1,\ldots,f_n}{g}\right) = \{x \in \Spa(A,A^+) \mid |f_i(x)| \le |g(x)| \ne 0~\text{for}~i=1,\ldots,n\},
\]
where $f_1,\ldots,f_n,g$ are elements of~$A$ generating an \emph{open} ideal of~$A$.
These subsets are called \emph{rational subsets}.
Note that if $A$ is a Tate ring, i.e. contains a topologically nilpotent unit~$\varpi$, the only open ideal is~$A$ itself because some power of~$\varpi$ would have to be contained in it.

If the elements $f_1,\ldots,f_n,g$ do \emph{not} generate an open ideal, the corresponding subset of $\Spa(A,A^+)$ is not rational (by definition) but still open.
This follows from the following exercise.

\begin{exercise} \label{generalized_rational}
 Let $f_1,\ldots,f_n,g$ be arbitrary elements of~$A$.
 For each $r \in \N$ let $E_r$ be a finite set of generators of $I^r$ (where~$I$ is an ideal of definition).
 Then
 \[
  R\left(\frac{f_1,\ldots,f_n}{g}\right) = \bigcup_{r \in \N} R\left(\frac{{\{f_1,\ldots,f_n\}} \cup E_r}{g}\right)
 \]
 and the set
 \[
  \{f_1,\ldots,f_n,g\} \cup E_r
 \]
 generates an open ideal of~$A$.
\end{exercise}

In Lemma~\ref{spectrum_completion} we have seen that $\Spa(A,A^+)$ is naturally isomorphic to $\Spa(\hat{A},\hat{A}^+)$.
A priori it is not clear, however, whether $\Spa(A,A^+)$ and $\Spa(\hat{A},\hat{A}^+)$ have the same rational subsets.
There could be more of them in $\Spa(\hat{A},\hat{A}^+)$ because we can consider rational subsets defined by $f_1,\ldots,f_n,g \in \hat{A}$ that are not all contained in $A$.
Luckily it turns out that we can approximate $f_1,\ldots,f_n,g$ by elements of~$A$ without changing the corresponding rational subset.
This is ensured by the following proposition.

\begin{proposition}
 Suppose $f_1,\ldots,f_n,g \in \hat{A}$ generate an open ideal.
 Then there is an open neighborhood $U$ of~$0$ in $\hat{A}$ such that for all $f'_i \in f_i + U$ and $g' \in g + U$ the ideal of $\hat{A}$ generated by $f'_1,\ldots,f'_n,g'$ is open and
 \[
  R\left(\frac{f_1,\ldots,f_n}{g}\right) = R\left(\frac{f'_1,\ldots,f'_n}{g'}\right).
 \]
\end{proposition}

\begin{proof}
 \cite{Hu93}, Lemma~3.10
\end{proof}

As $A$ is dense in $\hat{A}$ we can choose $f'_1,\ldots,f'_n,g'$ as in the lemma to be contained in~$A$ and then the corresponding rational subset~$U$ is defined by elements in~$A$.
However, this does not yet ensure that $f'_1,\ldots,f'_n,g'$ generate an open ideal \emph{in $A$}.
But we can extend the set $\{f'_1,\ldots,f'_n\}$ by elements $f'_{n+1},\ldots,f'_k$ that are less than or equal to $g'$ on~$U$ and such that the ideal of~$A$ generated by $f'_1,\ldots,f'_k,g$ is open (see \cite{Hu93}, Lemma~3.11).

\begin{example} \label{covering_unit_disc}
 Let~$k$ be a nonarchimedean field.
 For simplicity we assume it is algebraically closed.
 Then
 \[
  \DD_k=\Spa(k\langle T\rangle,k^\circ \langle T \rangle)
 \]
 is called the \emph{adic unit disc} over~$k$.
 It can also be described as the subset
 \[
  \{ x \in \Cont(k \langle T \rangle) \mid |T(x)| \le 1,~|k^\circ(x)| \le 1\}
 \]
 The second condition ensures that the restriction of~$x$ to $k$ coincides with the given valuation of~$k$.
 Without this condition $x|_k$ could be any continuous valuation of~$k$.
 As we have seen in Example~\ref{spectrum_affinoid_field} the corresponding valuation ring could be a proper subring of $k^\circ$ (this happens if the rank of $x|_k$ is greater than $1$).
 The condition $|k^\circ(x)| \le 1$ forces $k^\circ$ to be contained in the valuation ring implying that the valuation ring of $x|_k$ equals $k^\circ$.
 If $k=\CC_p$, this is automatic as its residue field is $\bar{\F}_p$, an algebraic extension of a finite field (see Example~\ref{spectrum_affinoid_field}).
 
 For $\alpha \in k^\circ$ and $r \in |k^\times|$ the rational subset
 \[
  \DD_k(\alpha,r) := \{x \in \DD_k \mid |(T-\alpha)(x)| \le r\}
 \]
 is a disc of radius~$r$ around~$\alpha$.
 Here~$\alpha$ is considered as a ``classical point'', i.e. the valuation
 \[
  k\langle T \rangle \overset{T \mapsto \alpha}{\longrightarrow} k \overset{|\cdot|}{\longrightarrow} \R_{>0} \cup \{0\}.
 \]
 The ``boundary ring''
 \[
  \SSS_k(0,1) := \{x \in \DD_k \mid |T(x)| = 1 \} = \{x \in \DD_k \mid |T(x)| \le 1, 1 \le |T(x)| \ne 0\}
 \]
 is also a rational subset.
 One might expect that the union of $\SSS_k(0,1)$ with all discs $\DD_k(0,r)$ of radius $r < 1$ is all of $\DD_k$.
 However the union misses one point:
 Consider the abelian group $\R_{>0} \times \R_{>0}$.
 We endow it with the lexicographical ordering, i.e.
 \[
  (a_1,b_1) < (a_2,b_2) \Leftrightarrow a_1 < a_2~\text{or}~(a_1 = a_2~\text{and}~b_1 < b_2).
 \]
 Now we pick an element $\epsilon \in \R_{>0} \times \R_{>0}$ of the form $(1,\epsilon_2)$ with $\epsilon_2 < 1$.
 So $\epsilon$ is an element that is less than $1 = (1,1)$ but infinitesimally close to~$1$ in the sense that $(r_1,r_2) < \epsilon$ whenever $r_1 < 1$.
 Via the embedding
 \[
  \R_{>0} \hookrightarrow \R_{>0} \times \R_{>0}, \qquad r \mapsto (r,1)
 \]
 we can consider the valuation on~$k$ as a valuation $k \to (\R_{>0} \times \R_{>0}) \cup \{0\}$.
 We now define a valuation
 \[
  x_{0,1-} : k\langle T \rangle \longrightarrow (\R_{>0} \times \R_{>0}) \cup \{0\}
 \]
 that maps $f = \sum_i a_i T^i \in k\langle T \rangle$ to
 \[
  |f(x_{0,1-})| = \max_i |a_i|\epsilon^i.
 \]
 This is indeed an element of $\DD_k$:
 It is continuous as for any pseudouniformizer~$\varpi \in k$ we have that $|\varpi(x_{0,1-})| = |\varpi|$ is cofinal in $\R_{>0} \times \R_{>0}$ and $|\varpi f(x_{0,1-})| < 1$ for any $f \in k^\circ\langle T \rangle$ (see the criterion for continuity in \cite{bergdall2024huber}, Proposition~1.2.7.1).
 Moreover $|f(x_{0,1-})| \le 1$ for all $f \in k^\circ \langle T \rangle$, so $x_{0,1-} \in \DD_k$.
 However,
 \[
  |T(x_{0,1-})| = \epsilon
 \]
 and
 \[
  r < \epsilon < 1~\forall~r \in [0,1).
 \]
 So neither is $x_{0,1-}$ contained in the ring $\SSS_k(0,1)$ nor in any $\DD_k(0,r)$, and this is indeed the only point of
$\DD_k$ not contained in this union, as follows from the classification of the points of $\DD_k$, which is explained in \cite{bergdall2024huber}, \S 1.4.4.
 
 We want to remark here that the point~$x_{0,1-}$ defined in this way does not depend on the choice of $\epsilon = (1,\epsilon_2)$ as long as $\epsilon_2 < 1$.
 The valuations for different choices of~$\epsilon$ are equivalent.
 This is why~$\epsilon$ does not show up in the index of $x_{0,1-}$. 
 By the minus sign we indicate that $\epsilon < 1$, or equivalently that we take a radius that is infinitesimally smaller than~$1$.
 If we choose $\epsilon >1$ we obtain a different valuation that will play a role in Example~\ref{closure_unit_disc}.
\end{example}

\section{The structure (pre-)sheaf}

We want to define on $X=\Spa(A,A^+)$ a sheaf of complete topological rings $\cO_X$ together with a subsheaf $\cO_X^+$ providing~$\cO_X$ with an integral structure satisfying
\[
 \cO_X(X) = \hat{A},	\qquad \cO_X^+(X) = \hat{A}^+
\]
The construction starts by assigning to each rational subset~$U$ a pair of rings $(A_U,A_U^+)$.
This defines two presheavs, $\cO_X$ and~$\cO_X^+$, on the category of rational subsets.
It is in the same spirit as the construction of the structure sheaf on an affine scheme.
Unfortunately there are Huber pairs $(A,A^+)$ such that $\cO_X$ does not satisfy the sheaf axioms.
The problem lies in the completion of rings that is involved in the construction.
So a priori we only construct a structure presheaf.
We will then discuss the main instances when $\cO_X$ is actually a sheaf.

For a rational subset
\[
 U = R\left(\frac{f_1,\ldots,f_n}{g}\right) = \{x \in X \mid |f_i(x)| \le |g(x)| \ne 0~\forall i=1,\ldots,n\}
\]
as above we consider the pair of rings
\begin{equation} \label{Huber_pair_rational}
 \left(A\left[\frac{1}{g}\right],A^+\left[\frac{f_1}{g},\ldots,\frac{f_n}{g}\right]^N\right)
\end{equation}
where $(\cdot)^N$ stands for taking the integral closure in $A_g = A[\frac{1}{g}]$ ($N$ for normalization).
We equip it with a topology in the following way.
Let $A_0$ be a ring of definition and~$I$ an ideal of definition for the Huber ring~$A$.
Then we turn~$A_g$ into a Huber ring by setting its ring of definition equal to
\[
 A_0\left[\frac{f_1}{g},\ldots,\frac{f_n}{g}\right]
\]
and its ideal of definition to
\[
 IA_0\left[\frac{f_1}{g},\ldots,\frac{f_n}{g}\right].
\]
One checks that with these definitions the pair of rings defined in (\ref{Huber_pair_rational}) is indeed a Huber pair.
We then take completions to obtain a complete Huber pair $(A_U,A_U^+)$.

\begin{lemma} \label{rational_subset_spectrum}
 The homomorphism of Huber pairs $(A,A^+) \to (A_U,A_U^+)$ induces an injection
 \[
 \Spa(A_U,A_U^+) \hookrightarrow \Spa(A,A^+)
 \]
 whose image is the rational subset~$U$.
\end{lemma}

Note that Lemma~\ref{rational_subset_spectrum} allows us to view the rational subset~$U$ as the adic spectrum of the Huber pair $(A_U,A_U^+)$.

\begin{proof}
 It is clear that the morphism is injective since~$A_U$ is the completion of~$A_g$ and the valuations of~$A_g$ form a subset of the valuations of~$A$.
 Also a valuation~$x$ of~$A$ extends to~$A_g$ if and only if $|g(x)| \ne 0$.
 Moreover~$x$ is less than or equal to~$1$ on
 \[
  A^+\left[\frac{f_1,\ldots,f_n}{g}\right]^N
 \]
 if and only if $|f_i(x)| \le |g(x)|$ for $i=1,\ldots,n$.
 This already shows that the image of $\Spa(A_U,A_U^+)$ is contained in~$U$.
 In order to show the converse inclusion, we need to check that for any $x \in U$ its extension to~$A_g$ is continuous.
 Remember that we defined
 \[
  \left(A_0\left[\frac{f_1,\ldots,f_n}{g}\right],IA_0\left[\frac{f_1,\ldots,f_n}{g}\right]\right)
 \]
 to be a pair of definition for $A_g$ for a given pair of definition $(A_0,I)$ of~$A$.
 By the criterion for continuity explained in \cite{bergdall2024huber}, Proposition~1.2.7.1, it is enough to convince ourselves that $|x(b)|$ is cofinal in the value group for $b \in I$ and $|x(a)| < 1$ for all $a \in IA_0[\frac{f_1,\ldots,f_n}{g}]$.
 This follows directly from the condition $x \in U$.
\end{proof}

We now define the presheaf~$\cO_X$ and its subpresheaf~$\cO_X^+$ on the category of rational subsets by
\[
 \cO_X(U) := A_U, \qquad \cO_X^+(U) := A_U^+.
\]
As mentioned before, there are examples where the sheaf condition is not satified (see Example~\ref{not_sheafy} below).
We call a Huber pair $(A,A^+)$ \emph{sheafy} if the structure presheaf $\cO_X$ on $X = \Spa(A,A^+)$ is a sheaf.

\begin{lemma}
 Let $(A,A^+)$ be a sheafy Huber pair and set $X = \Spa(A,A^+)$.
 Then $\cO_X^+$ is a sheaf.
\end{lemma}

\begin{proof}
 For a rational open covering $U = \bigcup_i U_i$ of a rational open $U \subseteq X$ we consider the diagram
 \[
  \begin{tikzcd}
   \cO_X(U)		\ar[r]							& \prod_i \cO_X(U_i)	\ar[r]							& \prod_{ij} \cO_X(U_i \cap U_j)	\\
   \cO_X^+(U)	\ar[r]	\ar[u,hookrightarrow]	& \prod_i \cO_X^+(U_i)	\ar[r]	\ar[u,hookrightarrow]	& \prod_{ij} \cO_X^+(U_i \cap U_j).	\ar[u,hookrightarrow]
  \end{tikzcd}
 \]
 The upper row is exact because $(A,A^+)$ is sheafy.
 It follows immediately that the left lower arrow is injective.
 Exactness in the middle is also clear once we know the following formula, which will be discussed in John's lecture (Theorem~3.3):
 \[
  \cO_X^+(U) = \{a \in \cO_X(U) \mid |a(x)| \le 1~\forall~x \in U\}.
 \]
\end{proof}

It turns out that in basically all known cases where we can show sheafiness of the structure presheaf the proof can be enhanced just a bit to show in addition that the structure sheaf is acyclic.
By this we mean explicitly that the sheaf cohomology groups
\[
 H^i(X,\cO_X)
\]
vanish for $i > 0$.
The following theorem summarizes the known results.
We will discuss them below and also explain all the adjectives showing up in the hypothesis.

\begin{theorem} \label{sheafy_results}
 Let $(A,A^+)$ be a Huber pair satisfying one of the following hypothesis.
 \begin{enumerate}[(i)]
  \item	The topology on~$A$ is discrete,
  \item	$A$ has a noetherian ring of definition,
  \item	$A$ is a strongly noetherian Tate ring,
  \item	$(A,A^+)$ is stably uniform.
 \end{enumerate}
 Then $(A,A^+)$ is sheafy and the structure sheaf is acyclic.
\end{theorem}

The easiest class of sheafy Huber pairs is the class of Huber pairs with discrete topology.
In this case there is no completion involved and the sheaf condition translates to the sheaf condition for fundamental opens in affine schemes.
So the proof of the theorem is straight forward (see \cite{Wed19}, Theorem~8.28~(c)).

The second case, where $A$ has a noetherian ring of definition, is often referred to as the ``formal scheme setting''.
The reason behind this nomenclature is that there is a fully faithful functor from locally noetherian formal schemes to adic spaces (see \cite{Hu94}, Proposition~4.2).
At the level of affine formal schemes it sends the formal spectrum $\Spf (A^+)$ to $\Spa(A^+,A^+)$.
The affinoids in the essential image of this functor satisfy condition~(ii) of Theorem~\ref{sheafy_results}.
Note that not all affinoids satisfying condition~(ii) come from formal schemes.
This is only the case for Huber pairs $(A,A^+)$ with $A = A^+$.
However, they are closely enough related to formal schemes to make the sheafiness proof for locally noetherian formal schemes work in this case (see \cite{Hu94}, Theorem~2.5).

We will now treat the case of Tate Huber pairs, which is most relevant to us.
Remember that a Huber ring is Tate if it contains a topologically nilpotent unit.
This is the case for the classes of Huber pairs described in (iii) and (iv) in Theorem~\ref{sheafy_results}.

Let us explain the terminology.
A Tate ring~$A$ is called \emph{strongly noetherian} if $A \langle T_1,\ldots,T_n \rangle$ is noetherian for any $n \ge 0$.
Looking at the condition for $n = 0$ this tells us in particular that $\hat{A}$ is noetherian.
An equivalent condition for strong noetherianness is that all $A$-algebras that are topologically of finite type are noetherian.
We call a Tate Huber pair $(A,A^+)$ strongly noetherian if $A$ is strongly noetherian.
Examples include all Tate rings of the form
\[
 k \langle T_1,\ldots,T_n \rangle /I
\]
for a nonarchimedean field~$k$ and any ideal $I \subseteq k \langle T_1,\ldots,T_n \rangle$.
This boils down to the classical result from rigid geometry that $k \langle T_1,\ldots,T_n \rangle$ is noetherian (\cite{BGR}, \S~5.2.6, Theorem~1).
The proof is by no means trivial and involves the Weierstraß division and preparation theorems.

A Tate ring~$A$ is \emph{uniform} if the subset of power bounded elements $A^{\circ}$ is bounded.
A Huber pair $(A,A^+)$ is called \emph{stably uniform} if for all rational subsets $U \subseteq X = \Spa(A,A^+)$ the Huber ring $\cO_X(U)$ is uniform.

There are strongly noetherian Tate Huber pairs that are not stably uniform and vice versa.
For instance for any nonarchimedean field~$k$ with pseudouniformizer~$\varpi$ the Huber ring
\[
 A:= k \langle T \rangle / (T^2)
\]
is strongly noetherian but not uniform.
The reason is that all elements $\varpi^{-n}T$ for $n \in \N$ are nilpotent, hence power bounded.
However, none of the sets
\[
 \left(\varpi^m k^\circ \langle T \rangle/ (T^2) \right) \cdot \{\varpi^{-n}T \mid n \in \N\}
\]
for $m \in \N$ is contained in the open $k^\circ \langle T \rangle /(T^2)$ as would be required if $\{\varpi^{-n}T \mid n \in \N\}$ were bounded.

In fact the same argument shows that more generally a Tate ring~$A$ with pair of definition $(A_0,(\varpi))$ is \emph{not} uniform as soon as it contains a nilpotent element $a$ such that
\[
 \{\varpi^n a \mid n \in \Z\} \nsubseteq A_0.
\]

On the other hand, a large class of non-noetherian stably uniform Tate Huber pairs is given by perfectoid Huber pairs $(A,A^+)$:

\begin{proposition} \label{perfectoid_uniform}
 Any perfectoid Huber pair $(A,A^+)$ is stably uniform.
\end{proposition}

\begin{proof}
 A perfectoid ring is uniform by definition (see Ben's notes, Definition~2.1.1).
 Moreover, by Theorem~3.1.4 in Ben's notes, for any rational open $U \subseteq \Spa(A,A^+)$ the ring $A_U$ is again perfectoid, hence uniform.
 We conclude that perfectoid Huber pairs are stably uniform.
\end{proof}

Let us get back to the question of sheafiness for Tate Huber pairs.
The following example shows that there are indeed non-sheafy Tate Huber pairs.

\begin{example}[\cite{BuzVer}, \S 4.1] \label{not_sheafy}
 Let~$k$ be a complete non-Archimedean field and
 \[
  A = k[T,T^{-1},Z]/(Z^2).
 \]
 We turn~$A$ into a Huber $k$-algebra by considering the ring of definition~$A_0$ generated as an $k^\circ$-module by
 \[
  \varpi^{|n|}T^n,	\qquad \varpi^{-|n|}T^nZ,	\quad \forall~n \in \ZZ.
 \]
 Then~$A_0$ is not noetherian as
 \[
  ZA \cap A_0 = (Z,\frac{Z}{\varpi T},\frac{Z}{\varpi^2 T^2},\ldots)
 \]
 is not finitely generated.
 
 We consider the cover of $X = \Spa(A,A^\circ)$ by the two opens
 \[
  U = \{x \in X \mid |T(x)| \le 1\}, \qquad V = \{x \in X \mid |T(x)| \ge 1\}.
 \]
 We claim that
 \[
  \cO_X(X) \to \cO_X(U) \oplus \cO_X(V)
 \]
 is not injective.
 We have
 \begin{align*}
  \cO_X(X) &= \hat{A} = \lim_n A/(\varpi^n A_0) \\
  \cO_X(U) &= \lim_n A/(\varpi^n A_0[T]) \\
  \cO_X(V) &= \lim_n A/(\varpi^n A_0[T^{-1}]).
 \end{align*}
 As $\varpi^{-n}Z \notin A_0$ for any $n \in \NN$, we know that $Z \ne 0$ in $\cO_X(X)$.
 But
 \begin{align*}
  \varpi^{-n}Z &= (\varpi^{-n}T^{-n}Z)T^n \in A_0[T], \\
  \varpi^{-n}Z &= (\varpi^{-n}T^nZ)T^{-n} \in A_0[T^{-1}].
 \end{align*}
 Hence, $Z = 0$ on~$U$ and on~$V$.
\end{example}

In Example~\ref{not_sheafy} even though~$A$ is noetherian, it is not strongly noetherian.
Neither is~$A$ uniform because $Z$ is nilpotent and $\varpi^{-m}Z$ for $m \ge 1$ is not contained in $A_0$.

Let us now outline the general strategy to study the
question of sheafiness for a Tate Huber pair.
As mentioned before we show sheafiness and acyclicity at the same time.
Let $X = \Spa(A,A^+)$ for a Tate Huber pair $(A,A^+)$.
For an open covering $U = \bigcup_i U_i$ for rational subsets~$U$ and~$U_i$ we consider the sequence
\begin{equation} \label{sequence_acyclicity}
 0 \to \cO_X(U) \to \prod_i \cO_X(U_i) \to \prod_{ij} \cO_X(U_i \cap U_j) \to \prod_{ijk} \cO_X(U_i \cap U_j \cap U_k) \to \ldots
\end{equation}
Exactness at the first two terms gives the sheaf condition for the covering $U = \bigcup_i U_i$ and exactness at all other terms is equivalent to the vanishing of the \u{C}ech cohomology groups
\[
 \check{H}^i(U,\cO_X)
\]
for $i > 0$.
It is a general fact (see \cite[Tag 01EV]{stacks-project}) that in order to show that the sheaf cohomology groups $H^i(X,\cO_X)$ vanish, it suffices to show that all \u{C}ech cohomology groups as above for all coverings $U = \bigcup_i U_i$ vanish.

In the proof for sheafiness the first step is to reduce to the case of very simple coverings.

\begin{definition}
 Let $(A,A^+)$ be a Tate Huber pair and set $X = \Spa(A,A^+)$.
 \begin{itemize}
  \item	Let $t_1,\ldots,t_n$ be elements of~$A$ generating the unit ideal.
		The \emph{standard rational covering} generated by $t_1,\ldots,t_n$ is the covering of~$X$ by
		\[
		 U_i = R\left(\frac{t_1,\ldots,t_n}{t_i}\right),
		 \quad\text{where } i\in\{1,\ldots,n\}.
		\]
  \item	Let $t_1,\ldots,t_n$ be in~$A$.
		The \emph{standard Laurent covering} generated by $t_1,\ldots,t_n$ is the covering of~$X$ by
		\[
		 U_I = \{x \in X \mid |t_i(x)| \le 1~\forall i \in I,~|t_i| \ge 1~\forall~i \notin I\},
		 \quad \text{where }I\subseteq\{1,\ldots,n\}.
		\]
  \item	A \emph{simple} Laurent covering is a standard Laurent covering generated by one element.
 \end{itemize}
\end{definition}

It is clear that Laurent coverings are indeed coverings.
In case of standard rational coverings this is ensured by the condition that $t_1,\ldots,t_n$ generate the unit ideal:
For $x \in X$ we take $t_i$ such that $|t_i(x)|$ is maximal among all $|t_j(x)|$.
Because $t_1,\ldots,t_n$ generate the unit ideal, we must have $|t_i(x)| \ne 0$.
By definition this implies $x \in U_i$.

\begin{lemma} \label{Laurent_cover}
 Let $(A,A^+)$ be a complete Tate Huber pair and $X = \Spa(A,A^+)$.
 \begin{enumerate}[(i)]
  \item Every covering of~$X$ has a refinement by a standard rational covering.
  \item	For every standard rational covering $\cU$ of~$X$ there exists a Laurent covering~$\cV$ of~$X$ such that for every $V \in \cV$ the covering $\cU_V:= \{V \cap U \mid U \in \cU\}$ is a standard rational covering generated by units.
  \item For every standard rational covering $\cU$ of~$X$ generated by units, there is a Laurent covering~$\cV$ of~$X$ refining~$\cU$.
 \end{enumerate}
\end{lemma}

\begin{proof}
 See \cite{Morel_adic_spaces}, Proposition~IV.2.3.3.
\end{proof}
 
Lemma~\ref{Laurent_cover} tells us that it is enough to show exactnes of sequence~(\ref{sequence_acyclicity}) for Laurent coverings.
Indeed if we start with some covering of~$X$, we can refine it by a standard rational covering by part~(i) of the lemma.
It is formal to check that exactness of sequence (\ref{sequence_acyclicity}) for a refinement of the covering implies exactness for the original covering.
We are thus reduced to deal with a standard rational covering~$\cU$.
By part~(ii) of Lemma~\ref{Laurent_cover} we can find a Laurent covering~$\cV$ of~$X$ such that for every $V \in \cV$ the restriction~$\cU_V$ is a standard rational covering of~$V$ generated by units.
Now we can convince ourselves that it is enough to prove exactness of sequence~(\ref{sequence_acyclicity}) for the Laurent covering~$\cV$ and for each of the restricted coverings~$\cU_V$.
But standard rational coverings that are generated by units can be refined by Laurent coverings by part~(iii) of the lemma.
We are thus reduced to Laurent coverings.
For a Laurent covering~$\cU$ generated by $t_1,\ldots,t_n$ we consider the simple Laurent covering $X = W_- \cup W_+$ with
\[
 W_- = \{x \in X \mid |t_n(x)| \le 1\},	\qquad W_+ =\{x \in X \mid |t_n(x)| \ge 1\}.
\]
The restrictions $\cU_{W_-}$ and $\cU_{W_+}$ are Laurent coverings generated by $t_1,\ldots,t_{n-1}$.
We can thus use induction to conclude that it is enough to prove exactness of sequence~(\ref{sequence_acyclicity}) for simple Laurent coverings.

We summarize our arguments in the following lemma.

\begin{lemma} \label{reduction_simple_Laurent}
 Let $(A,A^+)$ be a Tate Huber pair and set $X = \Spa(A,A^+)$.
 Suppose that for any rational open $W \subseteq \Spa(A,A^+)$ and any simple Laurent covering $W=W_- \cup W_+$, the sequence
 \begin{equation} \label{exactness_structure_sheaf}
  0 \to \cO_X(W) \to \cO_X(W_-) \oplus \cO_X(W_+) \to \cO_X(W_- \cap W_+) \to 0
 \end{equation}
 is exact.
 Then $(A,A^+)$ is sheafy and $\cO_X$ is acyclic.
\end{lemma}

We can now discuss case~(iii) of Theorem~\ref{sheafy_results}, where~$A$ is a strongly noetherian Tate ring.
The result goes back to Tate's acyclicity theorem first formulated in the language of Tate's rigid analytic spaces.
Huber gave a proof in the more general setting of strongly noetherian Tate Huber pairs (see \cite{Hu94}, Theorem~2.2) using the same approach.

\begin{proof}[Idea of the proof]
First note that being strongly noetherian is stable under passing to a rational open subset
\[
 W = R\left(\frac{f_1,\ldots,f_n}{g}\right) \subseteq \Spa(A,A^+).
\]
with $(f_1,\ldots,f_n) = A$.
The reason is that we can write $\cO_X(W)$ as
\[
 \cO_X(W) = A \langle T_1,\ldots,T_n \rangle / (g T_i - f_i).
\]
Therefore, we need to check the hypothesis of Lemma~\ref{reduction_simple_Laurent} only for $W= X = \Spa(A,A^+)$.
We consider a simple Laurent covering defined by an element $t \in A$, i.e., the covering of $X$ by
 \[
  W_- = \{x \in X \mid |t(x)| \le 1\}, \qquad W_+ = \{x \in X \mid |t(x)| \ge 1\}.
 \]
 We have
 \begin{align*}
  \cO_X(X)	 &= \hat{A} \\
  \cO_X(W_-) &= A\langle T \rangle /(T-t)	\\
  \cO_X(W_+) &= A\langle S \rangle /(tS-1)	\\
  \cO_X(W_- \cap W_+) &= A\langle S,T\rangle /(T-t, ST-1)
 \end{align*}
 The strongly noetherian assumption enters when showing that
 \[
  \hat{A} \to A\langle T \rangle /(T-t) \times A\langle S \rangle /(tS-1)
 \]
 is faithfully flat, hence injective.
 Showing the remaining exactness properties comes down to an explicit computation.
\end{proof}

Finally we want to explain the main argument for sheafiness in case~(iv) of Theorem~\ref{sheafy_results}, where $(A,A^+)$ is stably uniform.
The proof is due to Kevin Buzzard and Alain Verberkmoes and we refer to their article (\cite{BuzVer}, Proposition~12) for the details.

\begin{proof}[Idea of the proof]
 By definition any rational subset of $X = \Spa(A,A^+)$ is again stably uniform.
 This reduces us to checking the hypothesis of Lemma~\ref{reduction_simple_Laurent} for a simple Laurent covering of $W = X$.
 Explicitly this is a covering of the form
 \[
  X = \{x \in X \mid |t(x)| \le 1\} \cup \{x \in X \mid |t(x)| \ge 1\} =: W_- \cup W_+.
 \]
 We consider the following topological rings with topology given by the respective ring of definition and ideal of definition generated by $\varpi$:
 \begin{align*}
  A_- = A,			& \qquad (A_-)_0 = A_0[t]	\\
  A_+ = A[1/t],		& \qquad (A_+)_0 = A_0[1/t]	\\
  A_\pm = A[1/t],	& \qquad (A_\pm)_0 = A_0[t,1/t].
 \end{align*}
 (Here we are a bit sloppy because the homomorphims $\phi: A \to A[1/t]$ might not be injective, so we should write $(A_+)_0 = \phi(A_0)[1/t]$ etc.)
 Then
 \begin{align*}
 	\cO_X(X) 		&= \hat{A} \\
 	\cO_X(W_-) 		&= \widehat{A_-}	\\
 	\cO_X(W_+) 		&=\widehat{A_+}	\\
 	\cO_X(W_- \cap W_+)	&= \widehat{A_\pm}
 \end{align*}
 The sequence
 \[
 	0 \to A \overset{\epsilon}{\longrightarrow} A_- \oplus A_+ \overset{\delta}{\longrightarrow} A_\pm \to 0
 \]
 is obviously exact and completion gives the sequence~(\ref{exactness_structure_sheaf}).
 However, completion might destroy exactness if~$\epsilon$ and~$\delta$ are not strict (quotient topology on image coming from the source is the same as subspace topology coming from the target).
 The map $\delta$ is always strict because it is surjective and the image
 of $(A_-)_0\oplus(A_+)_0$ is $(A_\pm)_0$ and thus open; therefore the sequence (\ref{exactness_structure_sheaf}) is exact if and only if~$\epsilon$ is strict.
 The subspace topology on~$A$ induced from $A_- \oplus A_+$ is given by the ring of definition $S_0 = (A_-)_0 \cap (A_+)_0$ and strictness means that the sets $\varpi^n A_0$ are open with respect to the topology defined by~$S_0$.
 This boils down to checking that there is $m \in \N$ such that
 \[
  \varpi^m S_0 \subseteq A_0
 \]
 One can show that $S_0 \subseteq A^\circ$.
 By hypothesis $A$ is uniform, so $A^\circ$, hence $S_0$, is bounded.
 This implies the claim.
\end{proof}

\begin{corollary}
 Every perfectoid Huber pair $(A,A^+)$ is sheafy and the structure sheaf is acyclic.
\end{corollary}

\begin{proof}
 This follows from case~(iv) of Theorem~\ref{sheafy_results} since perfectoid Huber pairs are stably uniform by Proposition~\ref{perfectoid_uniform}.
\end{proof}

\section{Adic spaces}

In this section we explain how to glue adic spectra of sheafy Huber pairs to obtain adic spaces.
The process is similar to what is done for defining schemes.
Remember that a scheme is a locally ringed space that is locally isomorphic to the spectrum of a ring.
An adic space is also a locally ringed space but it has additional structure.
So instead of the category of locally ringed spaces we work in the category of locally $v$-ringed spaces, which is defined as follows.

\begin{itemize}
 \item	Objects are triples $(X,\cO_X,\{v_x\}_{x \in X})$, where $(X,\cO_X)$ is a locally topologically ringed space and for each $x \in X$ we have an equivalence class of continuous valuations $v_x$ on the local ring $\cO_{X,x}$ whose support is the maximal ideal.
 \item	Morphisms
		\[
		 (X,\cO_X,\{v_x\}_{x \in X}) \longrightarrow (Y,\cO_Y,\{v_y\}_{y \in Y})
		\]
		are morphisms of locally topologically ringed spaces $(X,\cO_X) \to (Y,\cO_Y)$ such that for a point $x \in X$ mapping to $y \in Y$ we have $v_x|_{\cO_{Y,y}} = v_y$.
\end{itemize}

The collection of valuations $\{v_x\}_{x \in X}$ determines a subsheaf $\cO_X^+$ of $\cO_X$ by setting for an open $U \subseteq X$
\[
 \cO_X^+(U) = \{a \in \cO_X(U) \mid v_x(a_x) \le 1~\forall~x \in U\}.
\]
Every sheafy Huber pair $(A,A^+)$ determines a locally $v$-ringed space by setting $X = \Spa(A,A^+)$ and considering the triple $(X,\cO_X,\{v_x\}_{x \in X})$, where the valuations~$v_x$ are determined by~$x$ in the following way.
For every rational open neighborhood $U \subseteq X$ of~$x$, $x$ comes from a unique valuation on $\cO_X(U)$ (see Lemma~\ref{rational_subset_spectrum}).
In the colimit over all neighborhoods we obtain the valuation~$v_x$ on $\cO_{X,x}$.
We call a locally $v$-ringed space arising in this way an \emph{affinoid adic space}.

A morphism of $v$-ringed spaces $(X,\cO_X,\{v_x\}_{x \in X}) \to (Y,\cO_Y,\{v_y\}_{y \in Y})$ is an \emph{open immersion} if it induces a homeomorphism of~$X$ onto an open subspace $U \subseteq Y$ and an isomorphism
\[
 (X,\cO_X,\{v_x\}_{x \in X}) \overset{\sim}{\longrightarrow} (U,\cO_Y|_U,\{v_u\}_{u \in U}).
\]
Inclusions of rational opens in an affinoid naturally give rise to open immersions of locally $v$-ringed spaces.

\begin{definition}
 An adic space is a locally $v$-ringed space $(X,\cO_X,\{v_x\}_{x \in X})$ that admits an open covering $X = \bigcup_i U_i$ by affinoid adic spaces.
\end{definition}

Adic spaces can be obtained from affinoid adic spaces by glueing just in the same way as for schemes.

\begin{lemma} \label{glue_adic_spaces}
 Suppose we are given
 \begin{itemize}
  \item an index set~$I$,
  \item for every $i \in I$ an adic space~$U_i$,
  \item for every $i,j \in I$ open subspaces $U_{ij} \subseteq U_i$ and $U_{ji} \subseteq U_j$ and an isomorphism $\varphi_{ij}:U_{ij} \to U_{ji}$
 \end{itemize}
 satisfying $\varphi_{ij}^{-1}(U_{ji} \cap U_{jk}) = U_{ij} \cap U_{ik}$ and the cocyle condition:
 For $i,j,k \in I$ the diagram
 \[
  \begin{tikzcd}
   U_{ij} \cap U_{ik}	\ar[rr,"\varphi_{ik}"]	\ar[dr,"\varphi_{ij}"']	&													& U_{ki} \cap U_{kj}	\\
																		& U_{ji} \cap U_{jk}	\ar[ur,"\varphi_{jk}"']
  \end{tikzcd}
 \]
 commutes.
 Then there is an adic space~$X$ and open immersions $\varphi_i:U_i \to X$ satisfying
 \begin{itemize}
  \item $X = \bigcup_i \varphi(U_i)$,
  \item $\varphi_i$ induce isomorphisms $U_{ij} \overset{\sim}{\to} \varphi_i(U_i) \cap \varphi_j(U_j)$,
  \item the diagram
		\[
		 \begin{tikzcd}
		  U_{ij}	\ar[rr,"\varphi_{ij}"]	\ar[dr,"\varphi_i"']	&	& U_{ji}	\ar[dl,"\varphi_j"]	\\
																	& X	\\
		 \end{tikzcd}
		\]
		commutes.
 \end{itemize}
\end{lemma}

\begin{proof}
 That the underlying topological space of $X$ can be glued from the~$U_i$ is clear. 
 As all morphisms are morphisms of locally $v$-ringed spaces, they respect the structure sheaf and the valuations.
 We thus obtain a structure sheaf on~$X$ and also a set of valuations $\{v_x\}_{x \in X}$.
 Since each~$U_i$ is locally isomorphic to an affinoid adic space, the same holds for~$X$.
\end{proof}

\begin{example}[The affine line] \label{example_affine_line}
 Let $k$ be a nonarchimedean field with pseudouniformizer $\varpi \in k$.
 For $n \in \NN$ the \textit{closed (adic) disc of radius}~$|\varpi^{-n}|$ is defined as
 \[
  \DD_k(0,|\varpi^{-n}|) = \Spa(k\langle \varpi^n T\rangle,k^\circ\langle \varpi^n T \rangle),
 \]
 where $\varpi^n T$ is to be treated as a variable.
 So actually the discs $\DD_k(0,|\varpi^{-n}|)$ are all isomorphic but we write them like this because of the way they are contained in each other.
 The notation is justified by looking at classical points:
 For $\alpha \in k$ the valuation
 \[
  x_\alpha: k[T] \overset{T \mapsto \alpha}{\longrightarrow} k \overset{|\cdot|}{\longrightarrow} \R_{\ge 0}
 \]
 extends to a continuous valuation on $k\langle \varpi^n T \rangle$ precisely if
 \[
  |\varpi^n T(x_\alpha)| \le 1  \iff  |\alpha| \le \varpi^{-n}.
 \]
 For $m \ge n$ we have natural homomorphisms
 \begin{equation} \label{transition_maps_discs}
  (k\langle \varpi^m T\rangle,k^\circ\langle \varpi^m T \rangle) \longrightarrow (k\langle \varpi^n T\rangle,k^\circ\langle \varpi^n T \rangle)
 \end{equation}
 mapping $\varpi^m T$ to $\varpi^{m-n}(\varpi^n T)$.
 The corresponding morphism of discs
 \[
  \DD_k(0,|\varpi^{-n}|) \to \DD_k(0,|\varpi^{-m}|)
 \]
 is precisely the inclusion of the disc of radius~$|\varpi^{-n}|$ into the disc of radius~$|\varpi^{-m}|$.
 In fact, $\DD_k(0,|\varpi^{-n}|)$ can be identified with the rational subset
 \[
  \{x \in \DD_k(0,|\varpi^{-m}|) \mid |T(x)| \le \varpi^{-n}\}.
 \]
 So we have an ascending sequence of discs whose radii become arbitrarily large.
 Since the inclusions are open immersions, we can glue the discs to obtain the \emph{affine line}:
 \[
  \A_k^1 = \bigcup_n \DD_k(0,|\varpi^{-n}|).
 \]
 As an adic space the affine line is not affinoid.
 This can be seen by noting that the discs $\DD_k(0,|\varpi^{-n}|)$ cover~$\A_k^1$ but this covering does not have a finite subcovering.
 However, all affinoid adic spaces are quasi-compact.
 
 Let us compute the global sections of the structure sheaf of $\A_k^1$.
 The sheaf property gives us an exact sequence
 \[
  \cO_{\A_k^1}(\A_k^1) \longrightarrow \prod_{m \in \N} \cO_{\A_k^1}(\DD_k(0,|\varpi^{-m}|)) \longrightarrow \prod_{m,n} \cO_{\A_k^1}(\DD_k(0,|\varpi^{-n}|) \cap \DD_k(0,|\varpi^{-m}|)).
 \]
 For $m \ge n$ the disc $\DD_k(0,|\varpi^{-n}|)$ is contained in $\DD_k(0,|\varpi^{-m}|)$.
 This simplifies the right hand term.
 Throwing out redundant factors we obtain an exact sequence
 \[
  \cO_{\A_k^1}(\A_k^1) \longrightarrow \prod_{m \in \N} \cO_{\A_k^1}(\DD_k(0,|\varpi^{-m}|)) \overset{\varphi}{\longrightarrow} \prod_{m > n} \cO_{\A_k^1}(\DD_k(0,|\varpi^{-n}|)),
 \]
 where $\varphi$ maps $(\alpha_m)_{m \in \N}$ to
 \[
  (\alpha_m|_{\DD_k(0,|\varpi^{-n}|)}-\alpha_n)_{m \ge n}
 \]
 The kernel of~$\varphi$ identifies with the projective limit
 \[
  \lim_{m \in \N} \cO_{A_k^1}(\DD_k(0,|\varpi^{-m}|))
 \]
 with transition maps given by (\ref{transition_maps_discs}).
 The rings $k \langle \varpi^m T \rangle$ are all subrings of the ring of formal power series $k \llbracket T \rrbracket$.
 Hence, the transition maps are injective and the above limit can be viewed just as an intersection taking place in $k \llbracket T \rrbracket$.
 
 Remember that
 \begin{align*}
  k \langle \varpi^m T \rangle	&= \left\{ \sum_{i=0}^\infty a_i \varpi^{im} T^i \in k \llbracket T \rrbracket \mid a_i \to 0\right\}	\\
  								&= \left\{ \sum_{i=0}^\infty b_i T^i \in k \llbracket T \rrbracket \mid b_i \varpi^{-im} \to 0\right\}
 \end{align*}
 are the power series converging when inserting $x \in k$ with $|x| \le |\varpi^{-m}|$.
 The intersection of these for all $m \in \N$ are the power series that converge everywhere.
 They are also called \emph{entire power series} and denoted by $k \{\{T\}\}$.
 Now what do they look like?
 Of course the polynomial ring is contained in $k \{\{T\}\}$.
 To construct an infinite series in $k \{\{T\}\}$, we just need to come up with coefficients that decrease very rapidly, faster than any polynomial.
 An example would be
 \[
  \sum_{i=0}^\infty \varpi^{i^2} T^i.
 \]
 Indeed, for any $m \in \N$
 \[
  \varpi^{i^2} \varpi^{-im} = \varpi^{i^2-im}
 \]
 converges to zero as $i$ goes to infinity since $\varpi$ is topologically nilpotent and $i^2 -im \to \infty$.
\end{example}

\begin{example}[The open unit disc] \label{open_unit_disc}
 The \emph{open unit disc} $\Dcirc_k(0,1)$ over a nonarchimedean field~$k$ is defined as the union of all closed discs with smaller radius taken inside $\DD_k(0,1)$ (or $\A_k^1$).
 If $\varpi \in k$ is a pseudouniformizer, we can write it as the union
 \[
  \Dcirc_k(0,1) = \bigcup_{m \in \N} \DD_k(0,|\varpi^{1/m}|)
 \]
 A priori $\DD_k(0,|\varpi^{1/m}|)$ is only well defined if $\varpi$ has an $m$-th root but we can make sense of it by defining it as the rational subset
 \[
  \DD_k(0,|\varpi^{1/m}|) := \{ x \in \DD_k(0,1) \mid |T^m(x)| \le |\varpi|\}
 \]
 of $\DD_k(0,1) = \Spa(k \langle T \rangle)$.
 We have seen in Example~\ref{covering_unit_disc} that $\Dcirc_k(0,1)$ is \emph{not} equal to the subset
 \[
  \{ x \in \DD_k(0,1) \mid |T(x)| < 1\}
 \]
 of $\DD_k(0,1)$.
 The above subset cointains one point more than $\Dcirc_k(0,1)$, namely the point $x_{0,1-}$ described in Example~\ref{covering_unit_disc}.
 
 Let us compute the global sections of the structure sheaf of $\Dcirc := \Dcirc_k(0,1)$.
 We start by determining the global sections of $\DD_k(0,|\varpi^{1/m}|)$.
 By definition $\cO_{\Dcirc}(\DD_k(0,|\varpi^{1/m}|))$ is the completion of $k[T]$ (or $k\langle T \rangle$) with respect to the topology defined by the pair of definition
 \[
  (k^\circ[T,T^m/\varpi],(\varpi)).
 \]
 The completion can be taken inside the power series ring $k \llbracket T \rrbracket$.
 To justify this first note that
 \[
  \cO_{\Dcirc}(\DD_k(0,|\varpi^{1/m}|)) = k \langle T, S \rangle /(\varpi S - T^m).
 \]
 We have a natural inclusion
 \[
  k \langle T, S \rangle \longrightarrow k \llbracket T, S \rrbracket
 \]
 In order to deduce an inclusion
 \[
  k \langle T, S \rangle/(\varpi S -T^m) \longrightarrow k \llbracket T, S \rrbracket/(\varpi S -T^m) = k \llbracket T \rrbracket
 \]
 we need to argue that
 \[
 \p:=(\varpi S -T^m) k \llbracket T \rrbracket
 \cap k \langle T,S \rangle
 \]
 equals the ideal of $k \langle T,S \rangle$ generated by $\varpi S - T^m$.\\
 First note that the polynomial $\varpi S - T^m$ is irreducible
 in $k\langle T,S\rangle$, which can easily be
 checked explicitly. Therefore it
 generates a prime ideal since $k\langle T,S\rangle$
 is a unique factorization domain (see \cite{BGR}, \S~5.2.6, Theorem~1), and it is clearly
 contained in $\p$. The equality follows if we
 show that $\p$ is not a maximal ideal, because
 $k\langle T,S\rangle$ has dimension $2$. But $k\langle T,S\rangle/\p$
 is not a field since it
 embeds into $k \llbracket T, S \rrbracket/(\varpi S -T^m) = k \llbracket T \rrbracket$
 and the image contains $T$, which is not a unit.

 Having identified the two ideals we get the desired inclusion
 \[
  \cO_{\Dcirc}(\DD_k(0,|\varpi^{1/m}|)) = k \langle T, S \rangle/(\varpi S -T^m) \longrightarrow k \llbracket T, S \rrbracket/(\varpi S -T^m) = k \llbracket T \rrbracket.
 \]
 It is now not hard anymore to identify the image of this embedding:
 \[
  \cO_{\Dcirc}(\DD_k(0,|\varpi^{1/m}|)) = \left\{\sum_{i=0}^\infty a_i T^i \mid
  a_i^m \varpi^i \to 0 \right\}
  =\left\{\sum_{i=0}^\infty a_i T^i\mid
  \lvert a_i\rvert\cdot\lvert\varpi\rvert^{i/m}\to
  0\right\}.
 \]
 Taking the intersection of them for varying $m \in \N$ yields the global sections of the open unit disc.
 They can be identified as the subring of the formal power series ring $k \llbracket T \rrbracket$ consisting of all power series that converge when inserting an element of $\bar{k}$ (the algebraic closure of~$k$) of norm less than $1$.
\end{example}

\begin{example}[The projective line]
 We want to explain how to glue two closed unit discs to obtain the projective line.
 Let us start with a nonarchimedean field~$k$ with pseudouniformizer~$\varpi$ and two discs
 \[
  \DD_1 := \Spa(k \langle T \rangle, k^\circ \langle T \rangle), \qquad \DD_2 := \Spa(k \langle S \rangle, k^\circ \langle S \rangle).
 \]
 We want to glue these two spaces along $\{|T|=1\}$ and $\{|S| = 1\}$ such that $T$ identifies with $S^{-1}$.
 Following Lemma~\ref{glue_adic_spaces} this is done by specifying a glueing isomorphism
 \[
  \SSS_1 := \{ x \in \DD_1 \mid |T(x)| =  1\} \overset{\sim}{\longrightarrow} \SSS_2 := \{ x \in \DD_2 \mid |S(x)| = 1\}.
 \]
 This can be done on the level of coordinate rings by writing down an isomorphism of the corresponding Huber pairs:
 \begin{align*}
  (k \langle S,S^{-1} \rangle, k^\circ\langle S,S^{-1} \rangle)	& \longrightarrow (k \langle T,T^{-1} \rangle,k^\circ\langle T,T^{-1}\rangle)	\\
  S																& \longmapsto T^{-1}.
 \end{align*}
 Since there are only two spaces to be glued, there is no cocycle condition to check.
 We call the resulting space the \emph{projective line} and denote it by $\P_k^1$.
 
 Let us compute the global sections of the structure sheaf.
 They are defined by the exact sequence
 \begin{align*}
  \cO_{\P_k^1}(\P_k^1) \longrightarrow k \langle T \rangle \oplus k \langle S \rangle	& \longrightarrow k \langle T,T^{-1} \rangle,	\\
															\left(\sum a_i T^i,\sum b_i S^i\right)	& \longmapsto \sum a_i T^i - \sum b_i T^{-i}
 \end{align*}
 coming from the sheaf condition.
 If
 \[
  \sum_{i=0}^\infty a_i T^i - \sum_{i=0}^\infty b_i T^{-i} = 0,
 \]
 all coefficients $a_i$ and $b_i$ for $i > 0$ have to vanish and we must have $a_0 = b_0$.
 We conclude that
 \[
  \cO_{\P_k^1}(\P_k^1) = k,
 \]
 just as for the scheme theoretic projective line.
 Of course we still need to justify the name \emph{projective line} for the space we have constructed.
 This will become more transparent in Section~\ref{section_analytification}, where we talk about analytification of algebraic varieties.
 
 Finally we want to explain how to embed the affine line into the projective line.
 For this we need to construct compatible empbeddings
 \[
  \DD_k(0,|\varpi^{-m}|) \longrightarrow \P_k^1.
 \]
 For $m=0$ we can just take the natural isomorphism
 \[
  \DD_k(0,1) \overset{\sim}{\longrightarrow} \DD_1.
 \]
 For $m > 0$ we cover $\DD_k(0,|\varpi^{-m}|)$ by two rational subsets
 \[
  \DD_k(0,|\varpi^{-m}|) = \DD_k(0,1) \cup \{ x \in \DD_k(0,|\varpi^{-m}|) \mid |T(x)| \ge 1\}.
 \]
 We map $\DD_k(0,1)$ isomorphically to $\DD_1$ via the natural map and define
 \[
  \{ x \in \DD_k(0,|\varpi^{-m}|) \mid |T(x)| \ge 1\} \longrightarrow \DD_2
 \]
 by the homomorphism of Huber pairs
 \begin{align*}
  (k \langle S\rangle,k^\circ\langle S \rangle)	& \longrightarrow (k\langle \varpi^m T, T^{-1}\rangle,k^\circ \langle \varpi^m T,T^{-1} \rangle	)\\
  S												& \longmapsto T^{-1}.
 \end{align*}
 We leave it to the reader to check that these maps are compatible and glue to an open immersion
 \[
  \iota: \A_k^1 \hookrightarrow \P_k^1.
 \]
 It is clear that $\DD_1$ lies in the image of~$\iota$ being the image of $\DD_k(0,1) \subseteq \A_k^1$.
 Moreover the image of $\DD_k(0,|\varpi^{-m}|)$ intersected with $\DD_2$ is the subset
 \[
  \{x \in \DD_2 \mid |S(x)| \ge |\varpi^m|\}
 \]
 Hence, the image of $\iota$ intersected with $\DD_2$ consists of all points $x \in \DD_2$ such that $|S(x)| > 0$ (remember that $\varpi$ is topologically nilpotent).
 If $|S(x)| = 0$, then $S$ lies in the support of the valuation~$x$, i.e. $x$ comes from a valuation of
 \[
  k \langle S \rangle /(S) = k
 \]
 Since $x$ needs to take values less than or equal to~$1$ on~$k^\circ$, there is only one such continuous valuation, namely the valuation of~$k$ we started with.
 This single point $x$ that does not lie in the image of~$\iota$ is a classical point of~$\DD_2$.
 It is the point $0 \in \DD_2$ to be precise.
 Since $\DD_1$ is glued to $\DD_2$ by mapping $T \to S^{-1}$, this corresponds to $T$ being equal to $\infty$ (loosely speaking).
 This is the reason why we call this point $\infty$.
\end{example}

\section{Local Huber pairs}

In the theory of schemes taking the intersection of all open neighborhoods corresponds to forming the localization at the corresponding prime ideal at the level of rings.
More precisely, for a point~$x$ of a scheme~$X$, we can describe the localization~$X_x$ of~$X$ at $x$ (which is defined as the intersection of all open neighborhoods of~$x$) in the following way.
Choose any open affine neighborhood $\Spec A$ of~$x$.
Then $x$ corresponds to a prime ideal $\p \subset A$ and $X_x = \Spec A_\p$.
So in particular, local schemes are spectra of local rings.

In this section we want to develop a similar correspondence for adic spaces.
Since the material is not thoroughly treated in the literature, we spell out all the proofs and technical details.
Things are a bit more subtle for adic spaces because there are two rings involved and we need to take into account the topology of the rings.
It is clear what we mean by the localization~$X_x$ of an adic space~$X$ at a point~$x$ just as a topological space: the intersection of all open neighborhoods of~$x$.
We will see that we can endow~$X_x$ with the structure of an affinoid adic space such that it is the limit of all open neighborhoods of~$x$ in an appropriate sense.
The corresponding Huber pair will be a \emph{local Huber pair}, so let us start by specifying what we mean by this.

If $A$ is a ring with a valuation~$v$, we can endow~$A$ with the valuation topology, which is the coarsest topology such that the subsets
\[
 \{a \in A \mid v(a) < \gamma\}
\]
are open for varying $\gamma \in \Gamma_v$.

\begin{definition}
 A Huber pair $(A,A^+)$ is called local\footnote{In \cite{HueAd}, Definition~6.1, appears a different definition of local Huber pair.
 The boundedness of $\m^+$ in \cite{HueAd} leads to~$A^+$ being a ring of definition, which might be too restrictive.
 Moreover, the topology in \cite{HueAd} might be coarser than the valuation topology leading to~$|\cdot|$ not being continuous.
 But then $\Spa(A,A^+)$ is empty, a case we want to exclude.}
 if the following conditions are satisfied:
 \begin{itemize}
  \item	$A$ and~$A^+$ are local,
  \item	$A^+$ is the preimage of a valuation ring~$k_A^+$ of the residue field~$k_A$ of~$A$,
  \item the valuation~$v_A$ of~$A$ defined by~$k_A^+$ is continuous.
 \end{itemize}
 A homomorphism $(A,A^+) \to (B,B^+)$ of local Huber pairs is local if $A \to B$ and $A^+ \to B^+$ are local ring homomorphisms.
\end{definition}

Actually the condition on~$A^+$ to be local follows from the second condition, that~$A^+$ be the preimage of~$k_A^+$.
We just included it to be able to speak of local homomorphisms $A^+ \to B^+$.

Let us take a look at the topological conditions in the definition.
The requirement that $v_A$ be continuous ensures that $v_A \in \Spa(A,A^+)$, otherwise the adic spectrum of $(A,A^+)$ would in fact be empty.
We can rephrase this condition by saying that the valuation topology defined by~$v_A$ needs to be coarser than the given topology on~$A$.

For what follows we fix a local Huber pair $(A,A^+)$.
It is a rather peculiar but very important property of local Huber pairs that the maximal ideal~$\m$ of~$A$ is contained in~$A^+$ and is indeed a prime ideal of~$A^+$.
This just follows from the fact that $A^+$ is the preimage in~$A$ of a subring of $k_A = A/\m$.
We also want to point out that
\[
 k_A^+ = A^+/\m.
\]
The valuation ring~$k_A^+$ defines a valuation~$|\cdot|$ on~$A$ such that
\[
 A^+ = \{a \in A \mid |a| \le 1\}.
\]

\begin{exercise}
 If $(A,A^+)$ is local, then $A$ is the localization of~$A^+$ at~$\m$.
\end{exercise}

The prime spectrum of~$A^+$ is of a very special form.
Since $A$ is a localization of~$A^+$, we already know that
\[
 \Spec A \hookrightarrow \Spec A^+.
\]
In order to understand the prime ideals of~$A^+$ that do not come from~$A$, the following lemma is crucial.

\begin{lemma} \label{ideals_A+}
 Let $(A,A^+)$ be a local Huber pair and $\a \subseteq A^+$ an ideal.
 Then either $\a \subseteq \m$ or $\m \subseteq \a$.
 If $\m \subsetneq \a$ and in addition~$\a$ is finitely generated, then~$\a$ is principal.
\end{lemma}

\begin{proof}
 Suppose that $\a \nsubseteq \m$.
 We want to show $\m \subseteq \a$, so let us take $m \in \m$ and try to show that $m \in \a$.
 By assumption there is $a \in \a$ which is not contained in~$\m$.
 Considered as an element of~$A$, $a$ is invertible and inside~$A$ we can form the product $a^{-1}m$.
 This is contained in the maximal ideal~$\m$, which in turn is contained in~$A^+$.
 We obtain
 \[
  m = a \cdot (a^{-1}m) \in \a.
 \]
 Thus we have $\m\subsetneq\a$.
 
 Now assume in addition that $\a$ is finitely generated, by $a_1,\ldots,a_n$, say.
 We consider the valuation $|\cdot|$ on $A$ induced by $k_A^+$.
 Choose $i$ such that $|a_i|$ is maximal. Since $\m\subsetneq \a$, the valuation $|a_i|$ must be greater than zero. In particular,
 $a_i$ is not contained in $\m$ and thus invertible in $A$.
 The elements $a_j/a_i\in A$, $j=1,\ldots,n$ all have valuation $\leq 1$ by construction, i.e. they lie in~$A^+$, so that $a_j\in a_iA^+$ for all $j$ and therefore $\a$ is principal.
\end{proof}

This lemma tells us that the spectrum of $A^+$ consists of the spectrum of~$A$ plus additional prime ideals strictly containing $\m$.
These prime ideals are totally ordered by inclusion and identify with the spectrum of $k_A^+$ (minus the zero ideal).

The topology of a Huber pair $(A,A^+)$ is particularly nice if $(A,A^+)$ is uniform.
Then~$A^+$ is bounded and thus a ring of definition.
Hence the results of \ref{ideals_A+} apply to any ideal of definition $I \subseteq A^+$.

In this setting we use the following notation.
If $I$ is contained in~$\m$, we speak of the \emph{formal case} (formal for formal scheme).
If in addition, $I = (0)$, we are in the \emph{discrete case}.
Otherwise, if $I \nsubseteq \m$ (i.e. $\m\subsetneq I$) we refer to this situation as the \emph{analytic case}.

We want to convince ourselves that in the analytic case the resulting topology of~$A$ coincides with the valuation topology.
Let us consider a uniform analytic local Huber pair $(A,A^+)$.
In particular, the corresponding valuation~$v_A$ is non-trivial or in other words $A^+ \ne A$.

\begin{exercise} \label{I_top_nil}
 Let $I \subseteq A^+$ be an ideal of definition.
 Then~$I$ is principal and generated by a topologically nilpotent element~$\varpi$.
\end{exercise}

Let~$A_{\val}$ be the abstract ring~$A$ endowed with the valuation topology.
Since~$v_A$ is continuous (in the topology of~$A$), the identity on~$A$ induces a continuous ring homomorphism
\[
 \psi: A \longrightarrow A_\val.
\]
The topologically nilpotent element~$\varpi$ from \ref{I_top_nil} maps to a topologically nilpotent element in~$A_\val$.
Put differently, $\varpi$ is also topologically nilpotent with respect to the valuation topology.
But this precisely means that the valuation~$v_A$ is microbial.

Now we can show that~$\psi$ is in fact a homeomorphism.
To achieve this we need to convince ourselves that the ideals $(\varpi^n)$ of $A^+$ are open in the valuation topology of~$A$.
But this follows from the identification
\[
 (\varpi^n) = \{ a \in A \mid v_A(a) \le v_A(\varpi^n)\}
\]
and Exercise~\ref{topology_valuation_ring}.
We summarize our results in the following proposition.

\begin{proposition}
 Let $(A,A^+)$ be a uniform analytic local Huber pair.
 Then~$A$ carries the valuation topology.
\end{proposition}

\begin{corollary}
 Let $(A,A^+)$ be a uniform local Huber pair.
 Then $(k_A,k_A^+)$ (with the induced topology) is an affinoid field.
\end{corollary}

%
%
%

\begin{lemma}
 The completion $(\hat{A},\hat{A}^+)$ of a uniform local Huber pair $(A,A^+)$ is local.
 In case $(A,A^+)$ is analytic, the completion $(\hat{A},\hat{A}^+)$ coincides with the completion of $(k_A,k_A^+)$ with respect to the valuation topology.
\end{lemma}

\begin{proof}
 If $(A,A^+)$ is formal, we first convince ourselves that also $(A,IA)$ is a pair of definition.
 The crucial point is that $I \subseteq \m$, whence
 \[
  IA \subseteq \m \subseteq A^+.
 \]
 This implies
 \[
  I^2 \subseteq I^2A \subseteq I,
 \]
 or in other words, the topology defined by $I^n$ coincides with the topology defined by $(IA)^n$.

 Now the assertion of the lemma (in the formal case) essentially follows from the fact that the $IA$-adic completion~$\hat{A}$ of the local ring~$A$ is local with the same residue field.
 The closure of the maximal ideal~$\m$ in~$\hat{A}$ is the maximal ideal~$\m_{\hat{A}}$ of $\hat{A}$.
 Since $\m_A \subseteq A^+$, this implies that $\hat{A}^+$ contains $\m_{\hat{A}}$.
 Moreover,
 \[
  \hat{A}^+/\m_{\hat{A}} \supseteq A^+/\m_A = k_A^+,
 \]
 hence $\hat{A}^+$ contains the preimage of~$k_A^+$ in $\hat{A}$.
 But this preimage is closed and contains~$A^+$, from which we conclude that it coincides with~$\hat{A}^+$.
 
 Let us now consider the analytic case.
 By Exercise~\ref{I_top_nil} we can take a pair of definition of the form $(A^+,(\varpi))$ for a topologically nilpotent unit~$\varpi$ of~$A$.
 Taking into account that $\varpi$ is not contained in~$\m$, Lemma~\ref{ideals_A+} tells us that all powers $(\varpi^n)$ for $n \in \N$ contain~$\m$.
 Therefore, the completion of $(A,A^+)$ factors through
 \[
  (A/\m,A^+/\m) = (k_A,k_A^+).
 \]
 The completion of $(k_A,k_A^+)$ we need is the one with respect to the topology defined by the pair of definition $(k_A^+,(\bar{\varpi}))$, where $\bar{\varpi}$ is the image of~$\varpi$ in~$k_A^+$.
 But we have seen in Exercise~\ref{topology_affinoid_field} that this topology coincides with the valuation topology.
 For the valuation topology we know that the completion~$\hat{k}_A$ of $k_A$ is again a valued field with valuation ring $\hat{k}_A^+$, which implies all claims.
\end{proof}

The denomination \emph{local} is justified for a local Huber pair by the following result.

\begin{proposition}
 A complete Huber pair $(A,A^+)$ with ring of definition~$A^+$ is local if and only if $\Spa(A,A^+)$ is local as a topological space (it has a unique closed point and all other points specialize to it).
 Furthermore, an adic space is local (as a topological space) if and only if it is the spectrum of a local Huber pair.
\end{proposition}

\begin{proof}
 Let $(A,A^+)$ be a local Huber pair.
 We claim that $\Spa(A,A^+)$ is local with closed point~$x$ given by the valuation ring $k_A^+ \subseteq k_A$.
 Since $\Spa(A,A^+)$ is spectral, it suffices to show that any open neighborhood of~$x$ equals $\Spa(A,A^+)$.
 In fact it is enough to check this for a rational open neighborhood
 \[
  U :=R\left(\frac{f_1,\ldots,f_n}{g}\right) = \{y \in \Spa(A,A^+) \mid |f_i(y)| \le |g(y)| \ne 0~\forall i=1,\ldots,n\}
 \]
 of~$x$ defined by elements $f_1,\ldots,f_n,g \in A$ generating the unit ideal.
 Since $x$ is supposed to be contained in~$U$, it satisfies
 \[
  |f_i(x)| \le |g(x)| \ne 0
 \]
 for $i=1,\ldots,n$.
 The condition $|g(x)| \ne 0$ translates to $g \notin \m_A$, so $g$ is invertible in~$A$.
 This implies that $v(g) \ne 0$ for \emph{any} valuation of~$A$.
 The inequalities can now be rearranged to
 \[
  \left|\frac{f_i}{g}(x)\right| \le 1,
 \]
 or in other words $f_i/g \in A^+$.
 But any point $y \in \Spa(A,A^+)$ satisfies $|a(y)| \le 1$ for all $a \in A^+$, in particular for $a=f_i/g$.
 We conclude that all points of $\Spa(A,A^+)$ are contained in~$U$.
 
 Conversely suppose that $(A,A^+)$ is a complete Huber pair such that $\Spa(A,A^+)$ is local with closed point~$x$.
 We first show that $A$ is local with maximal ideal $\m:= \supp x$.
 Suppose that~$\m' \subseteq A$ is a maximal ideal different from~$\m$.
 Since~$\m'$ is a maximal ideal of a complete Huber ring it is possible to construct a continuous valuation~$x'$ with support~$\m'$ lying in $\Spa(A,A^+)$ (see \cite{Hu94}, Lemma~1.4).
 But $x'$ does not specialize to~$x$.
 This can be seen by choosing an element~$g$ of~$\m'$ that is not contained in~$\m$.
 Then the open subset
 \[
  \{x \in \Spa(A,A^+) \mid |g(x)| \ne 0\}
 \]
 of $\Spa(A,A^+)$ contains~$x$ but not $x'$.
 To conclude, $x'$ does not specialize to~$x$, which gives a contradiction to our assumption that~$A$ is not local.
 Thus we now know that~$A$ is local with maximal ideal~$\m$.
 We denote its residue field by $k_A$.
 
 In order to prove that~$A^+$ is the preimage in~$A$ of a valuation ring of~$k_A$, it suffices to check the following:
 Every~$a \in A$ with $a \notin A^+$ is a unit in~$A$ and $1/a \in A^+$ (the reader can convince themselves that this criterion holds or consult \cite{KneZha}, Theorem~2.5).
 For such an~$a$ we consider the rational subset
 \[
  R\left(\frac{a}{1}\right) = \{ y \in \Spa(A,A^+) \mid |a(y)| \le 1\}.
 \]
 Its complement is nonempty as
 \[
  a \notin A^+ = \{b \in A \mid |b(y)| \le 1~\forall y \in \Spa(A,A^+)\}.
 \]
Since $\Spa(A,A^+)$ is local with closed point~$x$, this implies that~$x$ is contained in the complement, or in other words
\[
 |a(x)| > 1.
\]
First this implies $|a(x)| \ne 0$, whence $a \notin \m$, so $a \in A^\times$.
Then we get
\[
 \left|\frac{1}{a}(x)\right| < 1,
\]
so
\[
 x \in R \left(\frac{1}{a}\right) = \{ y \in \Spa(A,A^+) \mid 1\le |a(y)|\}
 =\left\{y\in\Spa(A,A^+)\mid \left|\frac{1}{a}(y)\right|
 \le 1\right\}.
\]
We conclude that this rational subset is all of $\Spa(A,A^+)$ and thus $1/a \in A^+$.

The condition on the topology of~$A$ follows automatically as~$x$ is a continuous valuation, which means that the valuation topology on~$A$ defined by~$x$ is coarser than the given topology.

Finally any local adic space needs to be affinoid as adic spaces are locally affinoid.
Therefore, the second assertion of the proposition follows from the first one. 
\end{proof}

We now have a basic understanding of local Huber pairs.
Our goal is to interpret $X_x$ for a point $x$ of an adic space~$X$ as the adic spectrum of a local Huber pair $(A,A^+)$.
This local Huber pair $(A,A^+)$ should morally be given by $(\cO_{X,x},\cO_{X,x}^+)$.
Unfortunately, $(\cO_{X,x},\cO_{X,x}^+)$ is not uniform, in general.
While it is true that $(\cO_{X,x},\cO_{X,x}^+)$ is local, we run into trouble when it comes to the topology.
The colimit topology on $\cO_{X,x}$ might be finer than the valuation topology of the valuation corresponding to~$x$.
But we don't want to use the colimit topology because it is quite messy and does not reflect the geometry of~$X_x$.

We can solve this issue by resorting to the concept of limit of adic spaces defined in \cite{Hu96}, Definition~2.4.2.
Following Huber an adic space~$X$ together with compatible maps
\[
 \pi_i : X \longrightarrow X_i
\]
is called a limit of a projective system of adic spaces $(X_i)_{i \in I}$ if the following two assertions hold.
\begin{enumerate}[(i)]
 \item	The morphisms~$\pi$ induce an identification $X = \lim_i X_i$ as topological spaces.
 \item	For every~$x \in X$ there is an open neighborhood $U \subseteq X$ such that the image of
 		\[
 		 \colim_{(i,V)} \cO_{X_i}(V) \longrightarrow \cO_X(U)
 		\]
		is dense, where $(i,V)$ runs over all pairs with $i \in I$ and $V$ an open subset of~$X_i$ containing the image of~$U$.
\end{enumerate}
We then write
\[
 X \sim \lim_{i \in I} X_i.
\]
Using ``$\sim$'' instead of ``$=$'' reflects the fact that this is not a projective limit in the categorical sense.
For that we would need the union in part (ii) to \emph{equal} $\cO_X(U)$.
But it turns out that the categorical limit is too restrictive in many situations and the above defined concept of limit serves our purposes.

Our strategy now is to write down a nice local Huber pair $(A,A^+)$, define
\[
 X_x := \Spa(A,A^+),
\]
and show that
\[
 X_x \sim \lim_{U} U,
\]
where the limit runs over all open neighborhoods~$U$ of~$x$.

We start with an affinoid open neighborhood $\Spa(B,B^+)$ of~$x$.
Let $\p \subseteq B$ denote the support of~$x$.
We define $A:=B_\p$ and consider the natural projection
\[
 \pi: A = B_\p \twoheadrightarrow k(\p) = k(x)
\]
The point~$x$ comes with a valuation of~$k(x)$ whose valuation ring we denote by $k(x)^+$.
Then we set
\[
 A^+:= \{a \in A \mid \pi(a) \in k(x)^+\}.
\]
We still need to specify a topology on~$A$.
Let $(B_0,J)$ be a pair of definition for~$B$ and assume $J \subseteq B^+$ (otherwise we replace~$J$ by a power).
We then take
\[
 I := JA^+ \subseteq A^+
\]
as an ideal of definition for~$A$ (with ring of definition~$A^+$).
This is the point where we possibly make the topology coarser than the colimit topololgy, in which $A^+$ might possibly not be bounded.

With this topology $(A,A^+)$ is a local Huber pair and
\[
 (B,B^+) \longrightarrow (A,A^+)
\]
is continuous.

\begin{exercise}
 If $(B,B^+) \to (B',B'^+)$ is a rational localization with $x \in \Spa(B',B'^+)$, we obtain a local homomorphism $(A,A^+) \to (A',A'^+)$ of the corresponding localizations at~$x$.
 This homomorphism induces an isomorphism on completions.
 \[
  (\hat{A},\hat{A}^+) \to (\hat{A}',\hat{A}'^+).
 \]
\end{exercise}

The above exercise shows that the completion of $(A,A^+)$ (and thus $\Spa(A,A^+)$) is independent of the choice of the affinoid neighborhood $\Spa(B,B^+)$ of~$x$.

\begin{proposition}
 Let $x$ be a point of an adic space~$X$ and $(A,A^+)$ the local Huber pair constructed above.
 Setting $X_x := \Spa(A,A^+)$ we have
 \[
  X_x \sim \lim_U U,
 \]
 where the limit runs over all open neighborhoods~$U$ of~$x$.
\end{proposition}

\begin{proof}
 Let us first verify that $X_x$ is a limit at the level of topological spaces.
 We consider the natural morphisms
 \[
  \Spa(A,A^+) \overset{\iota}{\longrightarrow} \Spa(B,B^+) \longrightarrow X,
 \]
 where $\Spa(B,B^+)$ is the affinoid open used in the construction of $(A,A^+)$.
 It is clear that~$\iota$ is injective as a valuation of $A = B_\p$ is uniquely determined by its restriction to~$B$.
 Moreover, its image is contained in the set of all generalizations of~$x$ as $\Spa(A,A^+)$ is local and its closed point maps to~$x$.
 
 Now let us pick a point $y$ specializing to~$x$ and show that it lies in the image of $\Spa(A,A^+)$.
 We start by checking that $\supp y$ is contained in $\p = \supp x$.
 For $b \in \supp y$ we consider the rational subset
 \[
  R\left(\frac{b}{b}\right) = \{ z \in \Spa(B,B^+) \mid |b(z)| \ne 0\}.
 \]
 If $x$ were contained in it, then also~$y$ would be contained but $|b(y)| = 0$.
 Therefore, $x$ is not contained, whence $|b(x)| = 0$, i.e. $b \in \p$.
 
 By what we have just seen the valuation~$y$ extends from~$B$ to~$A$.
 Moreover, $y$ takes values less or equal to~$1$ on~$A^+$ as $x$, hence~$y$, is contained in the rational subsets
 \[
  R\left(\frac{a}{1}\right) = \{ z \in \Spa(B,B^+) \mid |a(z)| \le 1\}
 \]
 for all $a \in A^+$.
 
 The critical point now is to check that~$y$ is a continuous valuation with respect to the topology of~$A$.
 Let $(B_0,I)$ be a pair of definition for~$B$.
 Then $(A^+,IA^+)$ is a pair of definition for~$A$.
 By \cite{bergdall2024huber}, Proposition~1.2.7.1, we need to check that $|b(y)|$ for $b \in I$ are cofinal in $\Gamma_y$ and that $|ba(y)| < 1$ for all $b \in I$ and $a \in A^+$.
 The first condition holds because~$y$ is a continuous valuation of~$B$ and the second condition because $|a(y)| \le 1$ for all $a \in A^+$.
 
 Finally we need to convince ourselves that $\iota$ induces a homeomorphism to the set of all generalizations of~$x$ (and not just a bijection).
 We consider the image via~$\iota$ of a rational open
 \[
  R\left(\frac{f_1,\ldots,f_n}{g}\right)
 \]
 of $\Spa(A,A^+)$.
 Multiplying $f_1,\ldots,f_n$ and~$g$ by a unit of~$A$ does not change the rational subset.
 Hence we can clear denominators by multiplying these elements with some $b \in B \setminus \p$ to achieve that they are all contained in the image of~$B$.
 We pick preimages $f'_1,\ldots,f'_n$ and $g'$ in~$B$.
 Then
 \[
  \iota(R\left(\frac{f_1,\ldots,f_n}{g}\right)) = R\left(\frac{f'_1,\ldots,f'_n}{g'}\right) \cap \{y \in \Spa(B,B^+) \mid y \rightsquigarrow x\}.
 \]
 The set
 \[
  R\left(\frac{f'_1,\ldots,f'_n}{g}\right) = \{ y \in \Spa(B,B^+) \mid |f'_i(y)| \le |g'(y)| \ne 0\}
 \]
 might not be rational because we cannot guarantee that $f'_1,\ldots,f'_n$ and~$g$ generate the unit ideal in~$B$.
 Nonetheless it is still open (see Exercise~\ref{generalized_rational}).
 This finishes the proof that $\iota$ is a homeomorphism onto the set of generalizations of~$x$.
 
 It remains to show that for any rational open $V = \Spa(C,C^+) \subseteq \Spa(A,A^+)$ the image of
 \[
  \colim_{\iota(V) \subseteq U} \cO_X(U) \longrightarrow (\hat{C},\hat{C}^+)
 \]
 is dense.
 We may assume that $C$ is a localization $A_f$ for some $f \in A$.
 We claim that in this case~$C$ lies in the image of the colimit.
 Indeed, an element $c$ of $C$ can be written as $a/f^n$ for $a \in A$ and $n \in \N$.
 In turn, $a$ can be written as $b/s$ for $b \in B$ and $s \in B \setminus \p$.
 This combines to
 \[
  c = \frac{b}{f^ns}.
 \]
 For all points $y \in V$ we have $|f^ns(y)| \ne 0$.
 Therefore,
 \[
  \iota(V) \subseteq \bigcup_{m \in \N} R\left(\frac{E_m}{f^ns}\right),
 \]
 where $E_m$ is a finite set of generators of~$I^m$.
 Since $\iota(V)$ is quasicompact, it is contained already in a finite union. 
 If we assume $B_0 \subseteq B^+$ (which we can always achieve by replacing $B_0$ by $B_0 \cap B^+$), the union is an ascending union and thus
 \[
  \iota(V) \subseteq R\left(\frac{E_m}{f^ns}\right)
 \]
 for some~$m$.
 The corresponding map on global sections of the structure sheaf is the completion of
 \[
  B\left[\frac{1}{f^ns}\right] \longrightarrow A_f = C,
 \]
 whose image obviously contains $c = b/(f^ns)$.
\end{proof}

\begin{remark}
 The Huber pair of local rings $(\cO_{X,x},\cO_{X,x}^+)$ is local but not necessarily uniform if it is equipped with the colimit topology.
 However, $X_x$ is uniform by construction.
 Forcing $\cO_{X,x}^+$ to be a ring of definition, $(\cO_{X,x},\cO_{X,x}^+)$ becomes a uniform local Huber pair and then
 \[
  X_x = \Spa(\cO_{X,x},\cO_{X,x}^+).
 \]
 In particular, the completion of $(\cO_{X,x},\cO_{X,x}^+)$ coincides with the completion of the Huber pair $(A,A^+)$ constructed above.
\end{remark}

In the following when we work with $(\cO_{X,x},\cO_{X,x}^+)$ we will always assume that $\cO_{X,x}$ carries the topology with pair of definition $(\cO_{X,x}^+,I\cO_{X,x}^+)$, where~$I$ is an ideal of definition coming from some affinoid open neighborhood of~$x$.
In particular,
\[
 X_x = \Spa(\cO_{X,x},\cO_{X,x}^+)
\]
and the global sections of the structure sheaf of $X_x$ are given by the completion
\[
 (\hat{\cO}_{X,x},\hat{\cO}_{X,x}^+)
\]

\section{Analytic adic spaces}

Analytic adic spaces are those one would say are closest to rigid analytic varieties but without finiteness assumptions.

\begin{definition}
 \begin{itemize}
  \item	Let $(A,A^+)$ be a Huber pair.
		A point $x \in \Spa(A,A^+)$ is \emph{analytic} if the prime ideal $\supp x$ of~$A$ is \emph{not} open.
  \item	A point $x$ of an adic space~$X$ is \emph{analytic} if it has an affinoid neighborhood $\Spa(A,A^+)$ such that $x$ is analytic as a point of $\Spa(A,A^+)$.
  \item	An adic space is \emph{analytic} if all its points are analytic.
 \end{itemize}
\end{definition}

Let us explain the intuition behind this definition.
We take a point $x$ of an adic space~$X$ and consider the (uniformized) local Huber pair $(\cO_{X,x},\cO_{X,x}^+)$.
Then we have the following:

\begin{exercise}
 With the above notation $x$ is analytic if and only if $(\cO_{X,x},\cO_{X,x}^+)$ is analytic or, in other words, the residue field $k(x)$ carries the valuation topology.
 Moreover, the points with support $\supp x$ correspond precisely to the continuous valuations of $k(x)$.
\end{exercise}

We learn from the exercise that only in the analytic case there is a continuity condition involved for valuations of $k(x)$.
Or in other words, for a Huber pair $(A,A^+)$ and an open prime ideal $\p$, \emph{any} valuation of the residue field $k(\p)$ defines a continuous valuation of~$A$.
In order to be an element of $\Spa(A,A^+)$ it just needs to satisfy the condition that it is less or equal to one on~$A^+$.

Let us get back to the study of analytic points and look at some examples.

\begin{example}
 $\Spa(\ZZ_p,\ZZ_p)$ has two points: the $p$-adic valuation on~$\QQ_p$ and the trivial valuation on $\F_p = \ZZ_p/p\ZZ_p$.
 The former point is analytic and the latter point is not.
\end{example}

The analyticity of a point $x \in X$ does not depend on the choice of an affinoid open neighborhood $\Spa(A,A^+)$ of $x$.
This is because for a rational subset $\Spa(B,B^+) \subseteq \Spa(A,A^+)$, the homomorphism $A \to B$ is adic, which is defined as follows.

\begin{definition}
 \begin{itemize}
  \item	A homomorphism of Huber rings $\varphi: R \to S$ is \emph{adic} if $R$ and~$S$ have rings of definition $R_0$ and~$S_0$ with $\varphi(R_0) \subseteq S_0$ such that for an ideal of definition $I_R \subseteq R_0$ the ideal $\varphi(I_R)S_0$ is an ideal of definition for~$S$.
  \item	A morphism of adic spaces is \emph{adic} if it is locally defined by an adic morphism of Huber pairs.
 \end{itemize}
\end{definition}

In other words one could say that for an adic homomorphism $R \to S$, the topology on~$S$ is defined by the topology on~$R$.
In particular, an ideal $J \subset S$ is open if and only if its preimage in~$R$ is open.
This settles the issue brought up above that analyticity is independent of the choice of affinoid open neighborhood.

We want to remark further that the critical point in the above definition is not that we can find rings of definition with $\varphi(R_0) \subseteq S_0$.
This is possible for any homomorphism of Huber rings.
The crucial point is that $\varphi(I_R)S_0$ is an ideal of definition.

\begin{proposition} \label{characterize_analytic}
 Let $x$ be a point of an adic space~$X$.
 The following are equivalent.
 \begin{enumerate}
  \item $x$ is analytic,
  \item $x$ has an affinoid open neighborhood $\Spa(A,A^+)$ such that~$A$ is a Tate ring (i.e. has a topologically nilpotent unit).
 \end{enumerate}
\end{proposition}

\begin{proof}
 Suppose~$x$ has an open affinoid neighborhood $\Spa(A,A^+)$ such that $\supp x \subseteq A$ is not open.
 Let $I = (a_1,\ldots,a_r)$ be an ideal of definition in a ring of definition~$A_0$ of~$A$.
 We can order the generators~$a_i$ in such a way that the value $|a_1(x)|$ is the maximum of all $|a_i(x)|$.
 This maximum cannot be zero as otherwise all $a_i$ would be contained in $\supp x$ and thus $\supp x$ would be open.
 Hence~$x$ is contained in the rational subset
 \[
  R\left(\frac{a_1,\ldots,a_r}{a_1}\right) = \{y \in \Spa(A,A^+) \mid |a_i(y)| \le |a_1(y)| \ne 0~\forall~i=1,\ldots,f\}.
 \]
 The corresponding Huber pair is the completion of
 \[
  (A_{a_1},A^+[\frac{a_2}{a_1},\ldots,\frac{a_r}{a_1}]^N),
 \]
 where $(\cdot)^N$ refers to taking the integral closure in $A_{a_1}$.
 Since $a_1 \in I$, it is topologically nilpotent.
 Its image in~$A_{a_1}$ is also topologically nilpotent and moreover, it is a unit.
 Therefore $A_{a_1}$ is a Tate ring.
 
 For the converse direction we just need to note that a Tate ring with topolologically nilpotent unit~$\varpi$ does not have open prime ideals as every open ideal has to contain the unit~$\varpi$.
\end{proof}

From this characterization we see that rigid analytic varieties (defined in the next section) and perfectoid spaces (discussed in Ben Heuer's lecture (\cite{HeuerPerfectoid}) are analytic.

\begin{lemma} \label{nonanalytic_to_nonanalytic}
 Let $f:X \to Y$ be a morphism of Huber pairs.
 Then~$f$ takes non-analytic points to non-analytic points.
\end{lemma}

\begin{proof}
 We may assume that~$f$ is given by a homomorphism of Huber pairs $\varphi:(A,A^+) \to (B,B^+)$.
 A non-analytic point~$x$ has open support $\supp x$.
 Its image in~$Y$ has support $\varphi^{-1}(\supp x)$, which is again open.
\end{proof}

It can happen, however, that a morphism takes analytic points to non-analytic points.
For instance, take a Tate ring~$R$ with ring of integral elements~$R^+$.
Let $R_\disc$ be the same ring~$R$ but endowed with the discrete topology and set $R_\disc^+ = R^+$.
Then the identity induces a homomorphism of Huber pairs $(R_\disc,R_\disc^+) \to (R,R^+)$.
Consider the corresponding morphism of adic spaces
\[
 \Spa(R,R) \to \Spa(R_\disc,R_\disc^+).
\]
All points of $\Spa(R,R^+)$ are analytic and all points of $\Spa(R_\disc,R_\disc^+)$ are non-analytic.
Something like this cannot happen for adic morphisms:

\begin{lemma} \label{analytic_to_analytic}
 Every adic morphism $f:X \to Y$ takes analytic points to analytic points.
\end{lemma}

\begin{proof}
 This follows from the observation that for an adic homomorphism of Huber pairs $\varphi:(R,R^+) \to (S,S^+)$ an ideal $J \subseteq S$ is open if and only if $\varphi^{-1}(J)$ is open.
\end{proof}

\begin{lemma} \label{analytic_adic}
 Every morphism of analytic adic spaces is adic.
\end{lemma}

\begin{proof}
 It suffices to consider a morphism induced by a homomorphism of Tate rings $\varphi:(R,R^+) \to (S,S^+)$ (see Proposition~\ref{characterize_analytic}).
 Let $R_0$ and $S_0$ be rings of definition with $\varphi(R_0) \subseteq S_0$.
 The ideal generated by any topologically nilpotent unit $\varpi$ of~$R$ which is contained in~$R_0$ is an ideal of definition of~$R_0$.
 But $\varphi(\varpi)$ is a topologically nilpotent unit of $S$ that is contained in~$S_0$, so it generates an ideal of definition.
\end{proof}

We learn from Lemma~\ref{nonanalytic_to_nonanalytic} and Lemma~\ref{analytic_to_analytic} that adic morphisms take analytic points to analytic points and non-analytic points to non-analytic points.
In particular, for an adic space~$X$ over $\Spa(\Z_p,\Z_p)$ whose structure morphism $X \to \Spa(\Z_p,\Z_p)$ is adic, the analytic points are those mapping to the $p$-adic valuation of~$\Q_p$ and the non-analytic points those mapping to the trivial valuation of~$\F_p$.
However, there are relevant examples of adic spaces over $\Spa(\Z_p,\Z_p)$ with non-adic structure morphism.

\begin{example} \label{ZpT}
 Consider the Huber ring $\Z_p\llbracket T \rrbracket$ with pair of definition $(\Z_p\llbracket T \rrbracket,(p,T))$.
 We want to study the adic space
 \[
  \cX:= \Spa(\Z_p\llbracket T \rrbracket) = \Spa(\Z_p\llbracket T \rrbracket,\Z_p\llbracket T \rrbracket)
 \]
 over $\Spa(\Z_p)$.
 Its structure morphism is not adic as the ideal of definition for $\Z_p\llbracket T \rrbracket$ has in addition to~$p$ the generator~$T$.
 Let us denote by~$\eta$ the point of $\Spa(\Z_p)$ corresponding to the $p$-adic valuation and by $s$ the point corresponding to the trivial valuation of~$\F_p$.
 As the notation suggests, $\eta$ is the generic point and~$s$ the closed point of $\Spa(\Z_p)$.
 Moreover, as noted before, $\eta$ is analytic and~$s$ is non-analytic.
 
 From Lemma~\ref{nonanalytic_to_nonanalytic} we know that all points mapping to~$\eta$ are supposed to be analytic.
 Indeed, the supports of these points are contained in the generic fiber of
 \[
  \Spec \Z_p\llbracket T \rrbracket \longrightarrow \Spec \Z_p,
 \]
 which is the set of all prime ideals \emph{not} containing~$p$.
 But such a prime ideal cannot be open as open prime ideals contain $(p,T)$.
 
 Among the points mapping to~$s$ there are many analytic points and only one non-analytic point.
 Namely, the support of a non-analytic point has to contain $(p,T)$.
 But $(p,T)$ is a maximal ideal with residue field $\F_p$, which has only the trivial valuation.
\end{example}

For an adic space~$X$ we denote by~$X_a$ the subset of analytic points and by $X_{na}$ the set of non-analytic points. Then Proposition \ref{characterize_analytic} immediately implies the following:

\begin{lemma}
 The set of analytic points~$X_a$ of an adic space~$X$ is open.
\end{lemma}

\begin{exercise}
 Let $X = \Spa(A,A^+)$ be an affinoid adic space,
 and fix
 a pair of definition $(A_0,I)$ as well as generators $s_1,\ldots,s_n$ for~$I$.
 Prove that
 \[
  X_a = \bigcup_{i=1}^n R\left(\frac{s_1,\ldots,s_n}{s_i}\right).
 \]
 In particular, for any quasicompact adic space $X$,
 the subset $X_a$ is even quasicompact open.
\end{exercise}

\begin{example} \label{ZpT_analytic_locus}
 Let us illustrate this by picking up Example~\ref{ZpT}.
 The analytic locus of $\cX=\Spa(\Z_p\llbracket T \rrbracket)$ is given by the union of two affinoids:
 \[
  \cX_a = R\left(\frac{p}{T}\right) \cup R\left(\frac{T}{p}\right).
 \]
 The rational subset $R(T/p)$ is contained in the generic fiber (the preimage of~$\eta$) as all of its points $x \in R(T/p)$ satisfy $|p(x)| \ne 0$.
 However, it does not cover all of the generic fiber as some points of the latter might satisfy $|p(x)| \ne 0$ but $|p(x)| < |T(x)|$.
 The generic fiber is open, having the following description
 \[
  \cX_\eta = \{x \in X \mid |p(x)| \ne 0\}.
 \]
 But it is not rational as the ideal generated by~$p$ is not open.
 We can cover it by rational open subsets in the following way:
 \[
  \cX_\eta=\bigcup_{n\geq 1}R\left(\frac{T^n}{p}\right).
 \]
 These sets are indeed rational as the ideals $(p,T^n)$ are open.
 Moreover, they cover~$\cX_\eta$ for the following reason.
 Since $T$ is topologically nilpotent, every open neighborhood of~$0$ in~$\Z_p\llbracket T \rrbracket$ contains a power of~$T$.
 This implies that for a continuous valuation~$v$ of $\Z_p\llbracket T \rrbracket$ with $v(p) \ne 0$ (i.e., any point $v$ of $\cX_\eta$), there is $n \in \N$ with
 \[
  v(T^n) < v(p).
 \]
 Thus the above union of rationals equals~$\cX_\eta$.
 
 At last we want to point out that we have come across the space~$\cX_\eta$ before.
 This might be more obvious when writing the rational subsets in the covering of~$\cX_\eta$ as spectra of Huber pairs:
 \[
  R\left(\frac{T^n}{p}\right) = \Spa\left(\Z_p\llbracket T \rrbracket\left[\frac{1}{p}\right], \Z_p\llbracket T \rrbracket\left[ \frac{T^n}{p}\right]\right).
 \]
 This might not yet look very familiar but it will once we have computed the completion of $(\Z_p\llbracket T \rrbracket[1/p], \Z_p\llbracket T \rrbracket[T^n/p])$.
 Note that a pair of definition is given by
 \[
  \left(\Z_p\llbracket T \rrbracket\left[\frac{T^n}{p}\right], (p,T)\right).
 \]
 Observing that
 \[
  T^n = \frac{T^n}{p} \cdot p \in (p),
 \]
 we realize that actually also $(p)$ is an ideal of definition.
 The $p$-adic completion of $\Z_p\llbracket T \rrbracket [T^n/p]$ can be computed as follows.
 Modulo $p^i$ every element of $\Z_p\llbracket T \rrbracket[T^n/p]$ has a representative which is a polynomial (of degree less than~$in$).
 Therefore, the completion of $\Z_p\llbracket T \rrbracket[T^n/p]$ is the same as the $p$-adic completion of
 \[
  \Z_p[T]\left[\frac{T^n}{p}\right] = \Z_p[T,S]/(pS-T^n),
 \]
 which equals
 \[
  \Z_p \langle T, S \rangle/(pS-T^n).
 \]
 Inverting~$p$ we get the completion of $\Z_p \llbracket T \rrbracket [1/p]$:
 \[
  \Q_p \langle T,S \rangle/(pS-T^n).
 \]
 Now we can recognize that
 \[
  R\left(\frac{T^n}{p}\right) = \Spa(\Q_p \langle T,S \rangle/(pS-T^n),\Z_p \langle T,S \rangle/(pS-T^n))
 \]
 is the disc of radius $p^{1/n}$ and the union over all $n \in \N$ gives the open unit disc:
 \[
  \cX_\eta = \bigcup_{n=1}^\infty \DD_{\Q_p}(0,p^{1/n}) = \Dcirc_{\Q_p}(0,1).
 \]
 See Example~\ref{open_unit_disc} for the discussion of the open unit disc.
\end{example}

\section{Morphisms of finite type}

In this section we explain what it means for a morphism of adic spaces to be of finite type.
It is too restrictive to require the corresponding ring maps to be of finite type.
Already the Tate algebra
\[
 k\langle T \rangle = \left\{\sum_{n \ge 0} a_n T^n \in k \llbracket T \rrbracket \mid a_n \to 0 \right\}
\]
over a nonarchimedean field~$k$ is not finitely generated but $k \langle T \rangle$ is supposed to be a topological analogue of the ring of polynomials $k[T]$.
Note that we can state the property of a ring homomorphism $R \to S$ to be of finite type by requiring~$S$ to be a quotient of a polynomial ring over~$R$ in finitely many variables.
This notion can be transported to the realm of Huber rings by considering topological quotients:

\begin{definition}
 A homomorphism $R \to S$ of complete Huber rings is a \emph{quotient map} if it is surjective, continuous, and open.
 A homomorphism $(R,R^+) \to (S,S^+)$ of Huber pairs is a \emph{quotient map} if $R \to S$ is a quotient map and $S^+$ is the relative integral closure of the image of~$R^+$ in~$S$.
\end{definition}

Any closed ideal $I \subseteq R$ defines a quotient map from $R$ to $S:=R/I$ equipped with the quotient topology.
We can enhance this to a quotient map of Huber pairs by defining~$S^+$ to be the integral closure of $R^+/(I \cap R^+)$ in~$S$.
We often denote this quotient Huber pair by $(R,R^+)/I$.

Now a straightforward replacement for ring homomorphisms of finite type in the topological setting is the following.

\begin{definition}
 A homomorphism $R \to S$ of Huber rings is \emph{strictly of topologically finite type} if there is a quotient map
 \[
  R\langle T_1,\ldots,T_n\rangle \twoheadrightarrow \hat{S}
 \]
 of $\hat{R}$-algebras.
\end{definition}

However, this class of $R$-algebras results to be too small to cover all important constructions.
Even an (algebraically) finitely generated $R$-algebra might not be a quotient of any Tate algebra $R \langle T_1,\ldots,T_n\rangle$.
In order to see this we take a look at rings of definition.

\begin{exercise} \label{ring_of_def_quotient}
 Let $R$ and~$S$ be Huber rings and
 \[
  \pi : R \langle T_1,\ldots,T_n \rangle \twoheadrightarrow S
 \]
 a quotient map.
 If $R_0 \subseteq R$ is a ring of definition, then $\pi(R_0 \langle T_1,\ldots,T_n \rangle)$ is a ring of definition of~$S$.
\end{exercise}

We can now give a simple example of a finite type $R$-algebra that is not strictly of topologically finite type.

\begin{example}
 We consider the Huber ring~$\Z_p$ and the $\Z_p$-algebra $\Q_p$.
 Since $\Q_p = \Z_p[1/p]$, it is finitely generated over~$\Z_p$.
 However, the following argument shows that it is not strictly of topologically finite type.
 Suppose there is a quotient map
 \[
  \Z_p \langle T_1,\ldots,T_n \rangle \twoheadrightarrow \Q_p
 \]
 The ring $\Z_p$ serves as its own ring of definition.
 It thus follows from Exercise~\ref{ring_of_def_quotient} that~$\Q_p$ is its own ring of definition, which is not true.
 Therefore, $\Q_p$ cannot be strictly of topologically finite type over~$\Z_p$.
 
 One can also see this more explicitly by checking that for any $a \in \Q_p \setminus \Z_p$ the homomorphism
 \begin{align*}
  \Z_p[T]	& \longrightarrow \Q_p	\\
  T			& \longmapsto a
 \end{align*}
 is \emph{not} continuous, where $\Z_p[T]$ is endowed with the $p$-adic topology.
 Otherwise we would get a homomorphism
 \[
  \Z_p \langle T \rangle \longrightarrow \Q_p
 \]
 mapping $T$ to~$a$.
 But $a^{-1} \in \Z_p$ and $a^{-1}T-1$ is contained in the kernel (it would in fact generate the kernel).
 However, in $\Z_p \llbracket T \rrbracket$ we have the identity
 \[
  \frac{1}{a^{-1}T-1} = \sum_{i=0}^\infty a^{-i} T^i
 \]
 and the series is in fact contained in $\Z_p \langle T \rangle$ as $|a| > 1$.
 Therefore, $a^{-1}T-1$ is invertible and generates the unit ideal.
 We would get a factorization
 \[
  \Z_p \langle T \rangle \longrightarrow \Z_p \langle T \rangle /(a^{-1} T -1) = 0 \longrightarrow \Q_p,
 \]
 a contradiction.
\end{example}

The example suggests that in addition to quotients of Tate algebras one should include quotients of more general $R$-algebras.
These generalized Tate algebras are constructed as follows.
We take finite subsets $M_1,\ldots,M_n$ of~$R$ such that for any $i$ and any $r \in \N$ the ideal
\[
 M_i^r R := \langle m_1 \cdots m_r a \mid m_1,\ldots,m_r \in M_i,~a \in R \rangle
\]
is open (the angle brackets stand for the subgroup generated by the listed elements).
Often we write~$M$ for the tuple $(M_1,\ldots,M_n)$ of finite sets.
We associate with each such tuple~$M$ of finite sets an $R$-algebra of \emph{weighted convergent power series}
\begin{align*}
 R \langle T \rangle_M	& := R \langle T_1,\ldots,T_n \rangle_M	\\
						& := \left\{\sum_{i \in \N^n} a_i T^i \in \hat{R}\llbracket T \rrbracket \mid \forall~0 \in U \subseteq \hat{R}~\text{open}:~a_i \in M^i U~\text{for almost all~$i$}\right\}
\end{align*}
Here we use the short notation~$T$ for the collection of the variables $T_1,\ldots,T_n$ and set
\[
 T^i := T^{i_1}\cdots T^{i_n}
\]
for $i \in \N^n$.
Likewise, for an open neighborhood $U \subseteq R$ of~$0$, we write
\[
 M^i U := M_1^{i_1} \cdots M_n^{i_n}U.
\]
With this notation we often replace the condition that $M_i^r R$ be open for all $i=1,\ldots,n$ and for all $r \in \N$ by the equivalent condition that $M^i R$ be open for all $i \in \N^n$.
Let us give this property a name:
We call a tuple $M = (M_1,\ldots,M_n)$ \emph{voluminous} if $M^i R$ is open for all $i \in \N^n$.

\begin{example}
 We can rewrite the Huber rings $k \langle \varpi^k T \rangle$ from Example~\ref{example_affine_line} as
 \[
  k \langle \varpi^k T \rangle = k \langle T \rangle_{\{\varpi^k\}}.
 \]
\end{example}

We need to give $R \langle T \rangle_M$ the structure of a Huber ring.
Starting with a pair of definition $(R_0,I)$ for~$R$ we specify the ring of definition
\[
 R_0 \langle T \rangle_M := \left\{\sum_{i \in \N^n} a_i T^i \in R \langle T \rangle \mid a_i \in M^i\hat{R}_0~\forall i \in \N\right\}
\]
and take the ideal of definition to be the ideal of $R_0 \langle T \rangle_M$ generated by~$I$.
Then for every open subgroup~$U$ of~$R$ the subgroup
\[
 U\langle T \rangle_M := \left\{\sum_{i \in \N^n} a_i T^i \in R \langle T \rangle \mid a_i \in M^i\hat{U}~\forall i \in \N^n\right\}
\]
is open in $R \langle T \rangle_M$.
In fact the collection of all $U \langle T \rangle_M$ as~$U$ varies over the open subgroups of $R$ is a fundamental system of open neighborhoods of zero.
Finally we define the ring of integral elements
\[
 R \langle T \rangle_M^+
\]
to be the integral closure of $R^+ \langle T \rangle_M$ in $R \langle T \rangle_M$.

\begin{definition}
 Let $R \to S$ be a homomorphism of Huber rings.
 Then $S$ is of \emph{topologically finite type} over~$R$ if there is a voluminous tuple $M = (M_1,\ldots,M_n)$ of finite subsets of~$R$ and a quotient map
 \[
  R \langle T \rangle_M \twoheadrightarrow S
 \]
 of $R$-algebras.
\end{definition}

\begin{example}
 Now luckily~$\Q_p$ is of topologically finite type over~$\Z_p$.
 Indeed
 \[
  \Q_p \cong \Z_p \langle T \rangle_{\{p\}} /(pT -1).
 \]
 The crucial point is that
 \[
  \frac{1}{pT-1} = \sum_{i=0}^\infty p^i T^i \in \Z_p \llbracket T \rrbracket
 \]
 is not contained in
 \[
  \Z_p \langle T \rangle_{\{p\}} = \left\{\sum_{i=0}^\infty a_i T^i \mid \forall r \in \N:~a_i \in p^{i+r}\Z_p~\text{for almost all}~i \in \N\right\}.
 \]
 Therefore $(pT-1)$ is not the unit ideal in $\Z_p \langle T \rangle$ and dividing out $(pT-1)$ results in inverting~$p$ in~$\Z_p$.
\end{example}

We get back the familiar ring of converging power series by choosing $M_i = \{1\}$ for all~$i$:
\[
 R \langle T_1,\ldots,T_n \rangle = R \langle T \rangle_{(\{1\},\ldots,\{1\})}.
\]
In particular, all $R$-algebras that are of strictly topologically finite type are also of topologically finite type.

It turns out that we need to consider the more complicated $R$-algebras $R \langle T \rangle_M$ only if we leave the realm of Tate rings.
This is not obvious by just looking at the definitions.
We need to take a detour via an alternative description of $R$-algebras (strictly) of topologically finite type.

\begin{proposition} \label{characterization_tft}
 For a homomorphism $\varphi: R \to S$ of Huber rings such that~$S$ is complete the following are equivalent:
 \begin{enumerate}[(i)]
  \item	$S$ is (strictly) of topologically finite type over~$R$,
  \item	There are rings of definition $R_0 \subseteq R$ and $S_0 \subseteq S$ with $\varphi(R_0) \subseteq S_0$ such that $R_0 \to S_0$ is strictly of topologically finite type and $S$ is finitely generated over (equal to) $\varphi(R) \cdot S_0$.
  \item	For every open subring $R_0 \subseteq R$ there is an open subring $S_0 \subseteq S$ with $\varphi(R_0) \subseteq S_0$ such that $R_0 \to S_0$ is strictly of topologically finite type and $S$ is finitely generated over (equal to) $\varphi(R) \cdot S_0$.
 \end{enumerate}
 Moreover, in the non-strict case we have the following equivalent characterization.
 \begin{enumerate}[(i)]
  \setcounter{enumi}{3}
  \item	$\varphi$ is adic, there are rings of definition $R_0 \subseteq R$ and $S_0 \subseteq S$ with $\varphi(R_0) \subseteq S_0$ and finite subsets $M \subseteq S$ and $N \subseteq S_0$ such that $\varphi(R)[M]$ is dense in~$S$ and $\varphi(R_0)[N]$ is dense in~$S_0$.
 \end{enumerate}
\end{proposition}

\begin{proof}
 We refer to \cite{Lud20}, Proposition~2.2 and~2.3.
\end{proof}

\begin{corollary}
 Any homomorphism $\varphi:R \to S$ of Tate rings of topologically finite type is strictly of topologically finite type.
\end{corollary}

\begin{proof}
 We may assume that~$S$ is complete.
 Then by Proposition~\ref{characterization_tft} there are rings of definition $R_0 \subseteq R$ and $S_0 \subseteq S$ with $\varphi(R_0) \subseteq S_0$ such that $R_0 \to S_0$ is strictly of topologically finite type and $S$ is finitely generated over $\varphi(R) \cdot S_0$.
 Let $\varpi \in R_0$ be a topologically nilpotent unit.
 Its image~$\varpi_S$ in~$S$ is a topologically nilpotent unit as well and moreover contained in~$S_0$.
 By \cite{bergdall2024huber}, \S 1.1.3, we have
 \[
  S = S_0 \left[\frac{1}{\varpi_S}\right] \subseteq \varphi\left(R_0\left[\frac{1}{\varpi}\right]\right) \cdot S_0 = \varphi(R) \cdot S_0,
 \]
 so $S$ is in fact strictly of topologically finite type over~$R$.
\end{proof}

Now we want to include the rings of integral elements into the picture.
A homomorphism of Huber rings $(R,R^+) \to (S,S^+)$ should be topologically of finite type if $S$ is topologically of finite type over~$R$ plus some additional finiteness condition on $R^+ \to S^+$.
There are several variants that play a role, which lead to different finiteness notions.
The following two are the most common:

\begin{definition}
 Let $\varphi: (R,R^+) \to (S,S^+)$ be a homomorphism of complete Huber pairs.
 Then $\varphi$ is
 \begin{itemize}
  \item \emph{weakly of topologically finite type} if $R \to S$ is of topologically finite type.
  \item of \emph{topologically finite type} if there is a quotient map
		\[
		 (R\langle T\rangle_M,R\langle T\rangle_M^+) \twoheadrightarrow (S,S^+)
		\]
		of Huber pairs over $(R,R^+)$ for some tuple~$M$ of finite subsets of~$R$ such that $M^i U$ is open for every open neighborhood~$U$ of zero in~$R$ and every $i \in \N^n$.
 \end{itemize}
\end{definition}

\begin{example}
 Let $(R,R^+)$ be a complete Huber pair.
 \begin{itemize}
  \item For a rational open $U \subseteq \Spa(R,R^+)$, the Huber pair $(R_U,R_U^+)$ is of topologically finite type over $(R,R^+)$.
  \item $(R\langle T \rangle,R^+\langle T \rangle)$ is of topologically finite type over $(R,R^+)$.
  \item Let $(R_\disc,R_\disc^+) = (R,R^+)$ but with the discrete topology.
		Then the identity induces a (continuous) homorphism of Huber pairs $(R_\disc,R_\disc^+) \to (R,R^+)$, which in general is \emph{not} of topologically finite type.
 \end{itemize}
\end{example}

After these simple examples let us take a look at a more sophisticated one that shows the difference between topologically finite type and weak topologically finite type.

\begin{example} \label{closure_unit_disc}
 Let $(k,k^\circ)$ be an analytic field with pseudouniformizer~$\varpi$.
 We consider the adic unit disc
 \[
  \DD_k = \Spa(k\langle T \rangle,k^\circ \langle T \rangle).
 \]
 As a subspace of~$\A_k^1$ (or a disc of bigger radius, for instance $\DD_k(0,\varpi^{-1})$) it is not closed even though it looks like a closed unit disc and it is even called by that name.
 It turns out there is one point in its closure that is not contained in $\DD_k$ itself:
 Here we return to the constructions we discussed in Example~\ref{covering_unit_disc}.
 We consider the ordered group $\R_{>0} \times \R_{>0}$ but this time choose $\epsilon = (1,\epsilon_2) \in \R_{>0} \times \R_{>0}$ with $\epsilon_2 > 1$.
 The corresponding valuation
 \begin{align*}
  x_{0,1+} : k\langle T \rangle	& \longrightarrow \R_{>0} \times \R_{>0}	\\
  \sum_i a_i T^i				& \longmapsto \max_i |a_i| \epsilon^i
 \end{align*}
 is not contained in $\DD_k$ as
 \[
  |T(x_{0,1+})| = \epsilon > 1.
 \]
 But it is still a continuous valuation of $k\langle T \rangle$ by Proposition~1.2.7.1 in \cite{bergdall2024huber}:
 We have that $|\varpi(x_{0,1+})| = |\varpi|$ is cofinal in the value group and for every element $\sum_i a_i T^i$ in the ring of definition $k^\circ \langle T \rangle$ we look at
 \[
  |\varpi\left(\sum_i a_i T^i\right)(x_{0,1+})| = \max_i |\varpi a_i| \epsilon^i.
 \]
 The terms $|\varpi a_i|$ are of the form $(r_1,1) \in \R_{>0} \times \R_{>0}$ with $r_1 < 1$.
 So the product looks like $(r_1,\epsilon_2^i) < 1$.
 
 The closure $\overline{\DD}_k$ of $\DD_k$ turns out to be an affinoid adic space that is weakly of topologically finite type but not of topologically finite type.
 We define
 \[
  A^+ = \left\{\sum_i a_i T^i \in k\langle T \rangle \mid |a_0| \le 1,~|a_i| < 1~\forall~i \ge 1\right\}.
 \]
 This is open as $\varpi k^\circ \langle T \rangle \subseteq A^+$, and let us check that it is integrally closed:
 If $x$ is integral over~$A^+$, it is contained in $k^\circ \langle T \rangle$.
 Its reduction $\bar{x}$ in
 \[
  \tilde{k}[T] = k^\circ \langle T \rangle/\m_k k^\circ \langle T \rangle
 \]
 ($\m_k$ is the maximal ideal of~$k^\circ$ and $\tilde{k}$ its residue field) is integral over
 \[
  A^+/(\m_k k^\circ\langle T \rangle) = \tilde{k}.
 \]
 Hence~$\bar{x} \in \tilde{k}$.
 By definition this implies $x \in A^+$.
 
 Now let us compute the valuation of $\sum_i a_i T^i \in A^+$:
 \[
  |\left(\sum_i a_i T^i\right)(x_{0,1+})| = \max_i |a_i| \epsilon^i.
 \]
 For $i = 0$ we obtain $|a_0| \le 1$ and for $i \ge 1$ we get $(|a_i|,\epsilon^i) < 1$, so $x_{0,1+} \in \Spa(k\langle T \rangle,A^+)$.
 In fact it is true that
 \[
  \overline{\DD}_k = \Spa(k\langle T \rangle,A^+).
 \]
 But $(k\langle T \rangle,A^+)$ is not of topologically finite type over $(k,k^\circ)$.
\end{example}

\begin{definition}
 A morphism of adic spaces $Y \to X$ is \emph{locally (weakly) of finite type} if locally on~$X$ and~$Y$ it is given by a morphism of affinoids induced by a homomorphism of Huber pairs that is (weakly) of topologically finite type.
 It is \emph{of (weakly) finite type} if in addition it is quasi-compact.
\end{definition}

\begin{proposition} \label{tft_adic}
 Every morphism of adic spaces that is locally weakly of finite type is adic.
\end{proposition}

\begin{proof}
 This follows from characterization~(iv) in Proposition~\ref{characterization_tft} of homomorphisms of Huber pairs of topologically finite type.
\end{proof}

\begin{definition}
 Let $(k,k^\circ)$ be an analytic field.
 A \textit{rigid analytic variety} over $(k,k^\circ)$ is an adic space~$X$ of finite type over $(k,k^\circ)$.
\end{definition}

For those familiar with the more classical language of Tate's rigid analytic varieties we want to remark the following.
There is a fully faithful functor from Tate's rigid analytic varieties to adic spaces sending $\Sp(R)$ to $\Spa(R,R^\circ)$.
Its essential image is (modulo some quasi-separatedness issues) given by the rigid analytic varieties defined above (see \cite{Hu94}, Proposition~4.4).

\section{Formal schemes and generic fibers} \label{formal_schemes_generic_fibers}

Let~$A$ be a noetherian ring endowed with the $I$-adic topology for a (finitely generated) ideal~$I$.
In particular, this is a Huber ring which is its own ring of definition.
Its \emph{formal spectrum} $\cX: =\Spf A$ is the set of all open prime ideals endowed with the Zariski topology.
So as a topological space $\cX$ is isomorphic to $\Spec A/I$.
What makes the difference is the structure sheaf~$\cO_{\cX}$:
For $f \notin \sqrt{I}$ its sections over the corresponding fundamental open $\Spf A_f \subseteq \cX$ are
\[
 \cO_{\cX}(\Spf A_f) = \widehat{A_f},
\]
where the completion is the $I$-adic completion.
In particular, $\cO_{\cX}(\cX) = \widehat{A}$.
Formal spectra $(\cX,\cO_{\cX})$ as above can be glued to obtain \emph{locally noetherian formal schemes}.
For details see \cite{FGA2}.

\begin{theorem}[\cite{Hu94}, Proposition~4.2]
 There is a fully faithful functor from the category of locally noetherian formal schemes to the category of adic spaces.
 It maps $\Spf A$ to $\Spa(A) = \Spa(A,A)$.
\end{theorem}

By the above theorem we can consider formal schemes as adic spaces.
More precisely we can identify them as the locally noetherian adic spaces~$X$ that can be covered by affinoid opens of the form $\Spa(A,A)$.

For a formal scheme~$\cX$ we write $\cX^{\ad}$ for the adic space associated to it according to the theorem.
Let us try to understand the relation between~$\cX$ and~$\cX^{\ad}$.
We will construct functorial morphisms of topologically locally ringed spaces
\[
 (\cX,\cO_{\cX}) \overset{\iota_\cX}{\longrightarrow} (\cX^{\ad},\cO_{\cX^{\ad}}^+) \overset{r_\cX}{\longrightarrow} (\cX,\cO_\cX)
\]
such that $r_\cX \circ \iota_\cX$ is the identity.
It is enough to give the construction in the affinoid case, i.e., where $\cX = \Spf(A)$.
Then~$\iota_\cX$ maps an open prime ideal $\p \subseteq A$ to the trivial valuation of $k(\p)$.

\begin{exercise} \label{preimage_iota}
 For a rational open subset $\Spa(B,B^+) \subseteq \cX^{\ad}$ we have
 \[
  \iota_\cX^{-1}(\Spa(B,B^+)) = \Spf B,
 \]
 where~$B$ is equipped with the $IB$-adic topology for an ideal of definition $I \subseteq B^+$.
 In particular, the preimage is empty if $IB = B$ is invertible in~$B$.
\end{exercise}

The exercise shows that~$\iota_\cX$ is continuous.
We now want to define the homomorphism of sheaves $\cO_{\cX^{\ad}}^+ \to (\iota_{\cX})_*\cO_\cX$.
Using Exercise~\ref{preimage_iota} this comes down to specifying for every open affinoid $\Spa(B,B^+)$ of~$\cX^{\ad}$ a homomorphism
\[
 B^+ = \cO_{\cX^{\ad}}^+(\Spa(B,B^+)) \longrightarrow \cO_\cX(\iota_\cX^{-1}(\Spa(B,B^+)) = \cO_\cX(\Spf B) = B.
\]
It is easy to check that by taking this homomorphism to be the natural inclusion, $\iota_\cX$ becomes a morphism of topologically locally ringed spaces.

In order to define the map~$r_\cX$ we consider a point $x \in \cX^{\ad}$, which defines a valuation ring $k(x)^+$ of $k(x)$.
Since $|A(x)| \le 1$, we have a natural map $A \to k(x)^+$ leading to
\[
 \Spec k(x)^+ \longrightarrow \Spec A.
\]
Then $r_\cX(x)$ is the image of the closed point of $\Spec k(x)^+$, also called the \emph{center} of~$x$.
Note that the continuity of~$x$ ensures that $r_\cX(x)$ is indeed an open prime ideal of~$A$.
Now it is not hard to check that for an affine open $\Spf B \subseteq \Spf A$ we have
\[
 r_\cX^{-1}(\Spf B) = \Spa(B,B).
\]
We define the sheaf homomorphism $\cO_\cX \to (r_\cX)_*\cO_{\cX^{\ad}}$ by taking over an open affinoid $\Spf B$ the identity of~$B$.
This defines~$r_\cX$ and obviously we have $r_\cX \circ \iota_\cX = \id$.

The existence of~$r_\cX$ and~$\iota_\cX$ suggests that $\cX$ should be thought of as a deformation retract of~$\cX^{\ad}$.
Although~$\cX^{\ad}$ has many more points, it carries more or less the same information as $\cX$.

Let $(k,k^\circ)$ be an analytic field.
We consider the special case of formal schemes over~$k^\circ$.
More precisely, we study the category of formal schemes that are locally of the form $\Spa(A,A)$ for a $k^\circ$-algebra $A$ of formally finite type.
For a noetherian ring~$A$ with the $I$-adic topology for some ideal $I \subseteq A$ the standard example of an $A$-algebra of formally finite type is
 \[
  \hat{A} \llbracket T_1,\ldots,T_n\rrbracket \langle X_1,\ldots,X_m \rangle
 \]
for $m, n \in \N \cup \{0\}$.
We endow this ring with the $J$-adic topology, where~$J$ is the ideal generated by~$I$ and $T_1,\ldots,T_n$.

\begin{exercise}
 Let~$A$ be a noetherian ring with an ideal $I \subseteq A$ and $m,n \in \N \cup \{0\}$.
 We consider the polynomial ring
 \[
  B := A[T_1,\ldots,T_n,X_1,\ldots,X_m]
 \]
 endowed with the $J$-adic topology, where~$J$ is the ideal generated by~$I$ and $T_1,\ldots,T_n$.
 Then
 \[
  \hat{B} = \hat{A} \llbracket T_1,\ldots,T_n\rrbracket \langle X_1,\ldots,X_m \rangle.
 \]
\end{exercise}

Note that $A \to B$ is adic if and only if $n=0$.
This in turn is the case if and only if~$B$ is of topologically finite type over~$A$.
Using the standard building blocks $\hat{A} \llbracket T \rrbracket \langle X \rangle$ we can now define $A$-algebras of formally finite type.

\begin{definition}
 A morphism $f : \cX \to \cY$ of locally noetherian formal schemes is of \emph{formally finite type} if it is locally on~$\cX$ and~$\cY$ given by a morphism of affinoids
 \[
  \Spa(B,B) \longrightarrow \Spa(A,A)
 \]
 induced by a ring homomorphism $A \to B$, such that there is a quotient map
 \[
  A \llbracket T_1,\ldots,T_n\rrbracket \langle X_1,\ldots,X_m \rangle \twoheadrightarrow \hat{B}
 \]
 of $\hat{A}$-algebras.
\end{definition}

In this setting an important construction is the (rigid analytic) generic fiber.
Classically, one would do this explicitly on affine opens and then check that the construction glues.
In the language of adic spaces the definition is straight forward:

\begin{definition}
 Let $(k,k^\circ)$ be an analytic field with pseudouniformizer $\varpi \in k^\circ$.
 Let $\cX$ be a formal scheme of formally finite type over $k^\circ$ considered as an adic space over $\Spa(k^\circ,k^\circ)$.
 The generic fiber of~$\cX$ is defined as the open subspace
 \[
  \cX_\eta = \{ x \in \cX | |\varpi(x)| \ne 0\}.
 \]
 If~$\cX$ is a (classical) formal scheme we write~$\cX_\eta$ for $(\cX^{\ad})_\eta$.
\end{definition}

Note that $\Spa(k^\circ,k^\circ)$ has two points.
The generic point is open and cut out precisely by the condition $|\varpi| \ne 0$.
Hence we see that~$\cX_\eta$ is nothing but the preimage in~$\cX$ of the generic point of $\Spa(k^\circ,k^\circ)$.

If $\cX$ is even of topologically finite type over~$k^\circ$, the inclusion of the generic fiber is quasi-compact.
In order to see this we may assume~$\cX$ is affinoid, i.e., $\cX = \Spa(A,A)$ for a $\varpi$-adic ring~$A$.
Then the ideal of~$A$ generated by~$\varpi$ is open, and thus
\[
 \cX_\eta = R \left(\frac{\varpi}{\varpi}\right)
\]
is a rational subset.

\begin{example} \label{generic_fiber_closed_disc}
 We consider the polynomial ring $k^\circ[T]$ over the valuation ring~$k^\circ$ of a nonarchimedean field~$k$ with the induced topology from~$k$.
 This is the $\varpi$-adic topology for some pseudouniformizer $\varpi \in k^\circ$.
 The completion of $k^\circ[T]$ is $k^\circ \langle T \rangle$.
 We want to determine the generic fiber of
 \[
  \cX := \Spa(k^\circ \langle T \rangle).
 \]
 By definition it is the rational subset
 \[
  R\left(\frac{\varpi}{\varpi}\right) = \{x \in \cX \mid |\varpi(x)| \ne 0\}.
 \]
 We obtain the corresponding Huber pair by inverting~$\varpi$ in the first component and adjoining $\varpi/\varpi=1$ in the second component (so nothing happens there):
 \[
  (k^\circ \langle T \rangle [\frac{1}{\varpi}],k^\circ \langle T \rangle[\frac{\varpi}{\varpi}]) = (k \langle T \rangle,k^\circ \langle T \rangle).
 \]
 The adic spectrum of this Huber pair is the closed unit disc $\DD_k(0,1)$.
\end{example}

\begin{example} \label{ZpT_generic_fiber}
 In Example~\ref{ZpT_analytic_locus} we have already computed the generic fiber of $\cX = \Spa(\Z_p\llbracket T \rrbracket)$:
 \[
  \cX_\eta = \bigcup_{n =1}^\infty R\left(\frac{T^n}{p}\right) = \Dcirc_{\Q_p}(0,1).
 \]
 Here one can see that $\cX_\eta$ is not quasi-compact due to the fact that $\Z_p \llbracket T \rrbracket$ is not of topologically finite type over~$\Z_p$.
\end{example}

Starting with a rigid analytic variety~$X$ over $(k,k^\circ)$ (i.e., an adic space locally of finite type over~$k$) we can always construct a formal scheme~$\cX$ over $\Spa(k^\circ,k^\circ)$ with $\cX_\eta = X$.
We call such a space~$\cX$ a \emph{formal model} of~$X$.
Formal models are almost never unique (only if~$X$ is quasi-finite over~$k$).
We can construct new formal models from an existing one - let us call it~$\cX$ -  by performing an \emph{admissible blowup}.
This is a blowup of~$\cX$ with respect to a sheaf of ideals~$\cI$ such that~$\varpi \in \cI$.
It turns out that admissible blowups are the only thing that stand in the way of uniqueness of formal models.

\begin{theorem}[Raynaud] \label{formal_model_equivalence}
 There is an equivalence of categories
 \[
  \left\{\substack{\text{qc $\varpi$-torsion free formal $k^\circ$-schemes of tft} \\ \text{localized by admissible formal blowups}}\right\}  \overset{\sim}{\longrightarrow} \big\{\substack{\text{qcqs rigid} \\ \text{$k$-spaces}}\big\}.
 \]
 induced by mapping a formal scheme~$\cX$ to its generic fiber~$\cX_\eta$.
\end{theorem}

The theorem lies the foundation for the formal model approach to nonarchimedean geometry.
Instead of rigid analytic spaces one can study formal schemes (up to admissible blowup).
This strategy has been adopted by many researchers, see \cite{BosLuetI}, \cite{FujKat06}.

Theorem~\ref{formal_model_equivalence} was originally formulated in the language of formal schemes and Tate's rigid analytic varieties.
In this setting the construction of the rigid analytic generic fiber is not so straight forward and one could say the generic fiber was missing some points.
In the language of adic spaces we can derive from Theorem~\ref{formal_model_equivalence} a more explicit version of the equivalence.
For a formal model~$\cX$ of a rigid analytic space~$X$ we construct the \emph{specialization map} as the composition
\[
 \spm_\cX: X = \cX_\eta \longrightarrow \cX^{\ad} \overset{r_\cX}{\longrightarrow} \cX.
\]
For an admissible blowup $\cX' \to \cX$ it is clear from the construction that the specialization maps $\spm_\cX$ and $\spm_{\cX'}$ are compatible.
We thus obtain a morphism
\[
 \spm : X \longrightarrow \lim_{\cX_\eta=X} \cX,
\]
where the limit runs over all formal models~$\cX$ of~$X$.

\begin{theorem}
 Let~$X$ be a qcqs rigid analytic space over~$k$.
 Then the morphism $\spm$ induces an isomorphism
 \[
  \spm: (X,\cO_X^+) \overset{\sim}{\longrightarrow} \lim_{\cX_\eta = X} (\cX,\cO_{\cX})
 \]

\end{theorem}

\begin{proof}
 For an element~$x$ of $\lim \cX$ we denote by $\cO_{\cX,x}$ the local ring of the formal model~$\cX$ at the image point of~$x$.
 The crucial point is to show that for each such~$x$ the $\varpi$-adic completion of
 \[
  \colim_{\cX} \cO_{\cX,x}
 \]
 is a microbial valuation ring with pseudouniformizer~$\varpi$.
 This is shown in \cite{Bhatt17}, Proposition~8.1.3.
 We call this valuation ring~$R_x$.
 For every formal model and open affine neighborhood $\Spf B$ of~$x$ we have a homomorphism
 \[
  B \longrightarrow R_x
 \]
 defining a valuation~$v_x$ of~$B$.
 Since~$R_x$ is microbial with pseudouniformizer~$\varpi$, this valuation is continuous and analytic.
 Hence it defines a point of the generic fiber $\cX_\eta = X$.
 Finally one needs to check that this construction gives an inverse to~$\spm$.
\end{proof}

\section{The fiber product of adic spaces}

Let $Y \to X$ and $Z \to X$ be morphisms of adic spaces.
The \emph{fiber product} of these two morphisms is an adic space $Y\times_X Z$ fitting into a diagram
\[
 \begin{tikzcd}
  Y \times_X Z	\ar[r]	\ar[d]	& Y	\ar[d]	\\
  Z				\ar[r]			& X
 \end{tikzcd}
\]
satisfying the usual universal property:
For every solid arrow diagram below there is a unique dotted arrow making the diagram commute:
\[
 \begin{tikzcd}
  S	\ar[drr,bend left]	\ar[dr,dashed]	\ar[ddr,bend right]	& [-25pt]	\\[-12pt]
															& Y_X Z		\ar[r]	\ar[d]	& Y	\ar[d]	\\
															& Z			\ar[r]			& X
 \end{tikzcd}
\]
For topological issues fiber products do not exist in general.
We need to impose some finiteness assumptions on $Y \to X$ and $Z \to X$.
Moreover we have to make sure the sheaf condition for the structure presheaf is satisfied on the fiber product.
As in the case of schemes the construction of the fiber product involves taking tensor products of the corresponding rings.
However, it is a bit more complicated and the tensor product of affinoids need not be affinoid but rather an ascending union of affinoids.

\begin{definition}
 A Huber ring~$A$ is \emph{stably sheafy} if for every $A$-algebra~$B$ of topologically finite type and every ring of integral elements $B^+ \subseteq B$ the Huber pair $(B,B^+)$ is sheafy.
 An adic space is \emph{stable} if it is locally the spectrum of stably sheafy Huber pairs.
\end{definition}

The classes of sheafy Huber pairs we have treated so far - those listed in Theorem~\ref{sheafy_results} - are all stably sheafy because their defining property is preserved when passing to an adic space of topologically finite type over them.
Hence, all our familiar adic spaces are stable.

\begin{theorem} \label{fiber_product_existence}
 Let~$X$ be a stable adic space and $f:Y \to X$ and $g:Z \to X$ morphisms of stable adic spaces satisfying one of the following conditions
 \begin{enumerate}[(i)]
  \item $X$, $Y$, and~$Z$ are perfectoid,
  \item $f$ is locally of weakly finite type and $g$ is adic.
  \item $f$ is locally of finite type,
 \end{enumerate}
 Then the fiber product $Y \times_X Z$ exists in the category of adic spaces and is a stable adic space.
 In case~(i) $Y \times_X Z$ is perfectoid.
\end{theorem}

\begin{proof}[Idea of proof]
 As for the fiber product of schemes we reduce to the case of affinoid adic spaces:
 \[
  X = \Spa(A,A^+), \qquad Y = \Spa(B,B^+), \qquad Z = \Spa(C,C^+).
 \]
 Moreover, we assume that~$f$ and~$g$ come from homomorphisms of Huber pairs:
 \[
  \begin{tikzcd}
			& (B,B^+)	\\
   (C,C^+)	& (A,A^+)	\ar[u,"\varphi_B"']	\ar[l,"\varphi_C"].
  \end{tikzcd}
 \]
 Case~(i) and~(ii) are easier because all morphisms are adic:
 Perfectoid spaces are analytic and all morphisms of analytic adic spaces are adic (Lemma~\ref{analytic_adic}).
 In case~(ii) $\varphi_B: A \to B$ is topologically of finite type, hence adic (Proposition~\ref{tft_adic}) and $A \to C$ is adic by assumption.
 
 So there are rings of definition $A_0 \subseteq A$, $B_0 \subseteq B$, and $C_0 \subseteq C$ with
 \[
  \varphi_B(A_0) \subseteq B_0, \qquad \varphi_C(A_0) \subseteq C_0
 \]
 such that for an ideal of definition $I_A \subseteq A_0$ the ideals
 \[
  \varphi_B(I_A)B_0 \subseteq B_0, \qquad \varphi_C(I_A)C_0 \subseteq C_0
 \]
 are ideals of definition.
 We set
 \[
  D = B \otimes_A C.
 \]
 We define its ring of definition~$D_0$ as the image of $B_0 \otimes_{A_0} C_0$ in~$D$ and let $D^+$ be the integral closure of $B^+ \otimes_{A^+} C^+$ in~$D$.
 Then $(D,D^+)$ is a Huber pair whose ideal of definition is generated by the image of~$I_A$ in~$D_0$.
 
 The additional assumptions are needed to ensure that $(D,D^+)$ is sheafy.
 In the first case one shows that $(D,D^+)$ is perfectoid, so it is sheafy.
 In the second case we use that $C \to D$ is a basechange of $A \to B$, which is topologically of finite type.
 Since $Y = \Spa(C,C^+)$ is stable, this ensures that $(D,D^+)$ is sheafy.
 
 It then turns out that
 \[
  \Spa(D,D^+) = Y \times_X Z.
 \]
 For details see \cite{Hu96}, 1.2.2.
 
 Let us now treat case~(iii).
 We may assume that
 \[
  (C,C^+) = (A \langle T \rangle_M,A \langle T \rangle_M^+)/I
 \]
 for a voluminous tuple $M = (M_1,\ldots,M_n)$ of finite subsets of~$A$ and a closed ideal $I \subseteq A \langle T \rangle_M$.
 Let~$A_0$ and $B_0$ be rings of definition with $\varphi_B(A_0) \subseteq B_0$.
 Now the problem is that $\varphi_B(M)$ might not be voluminous because $\varphi_B$ might not be adic.
 We can resolve this by adding additional elements.
 Let $L$ be a set of generators for an ideal of definition of $B$.
 For every $i \in \N$ we define the tuple
 \[
  M_{L,i} := (\varphi_B(M_1) \cup L^i,\ldots,\varphi_B(M_n) \cup L^i)
 \]
 of finite subsets of~$B$.
 Since $L$ generates an ideal of definition, the tuples $M_{L,i}$ are voluminous for every $i \in \N$.
 Hence we can consider the weighted Tate algebras
 \[
  B \langle T \rangle_{M_{L,i}}.
 \]
 For every $i > j$ the universal property of the weighted Tate algebra (see \cite{Hu94}, Lemma~3.5~(i)) gives us a canonical morphism
 \[
  B \langle T \rangle_{M_{L,i}} \longrightarrow B \langle T \rangle_{M_{L,j}}.
 \]
 Likewise we have compatible homomorphisms
 \[
  A \langle T \rangle_M \longrightarrow B \langle T \rangle_{M_{L,i}}.
 \]
 Let $J_i$ be the closure of the ideal of $B \langle T \rangle_{M_{L,i}}$ generated by~$I$.
 Then we get an injective system of $A$-algebras $B \langle T \rangle_{M_{L,i}}/J_i$ fitting into the diagram
 \[
  \begin{tikzcd}
   {}	\ar[dr,dotted]	\\
						& B \langle T \rangle_{M_{L,2}}/J_2	\ar[dr]	\\
						&											& B \langle T \rangle_{M_{L,1}}/J_1				& B	\ar[l] \ar[ull]	\\
						&											& A \langle T \rangle_M/I			\ar[u]	\ar[uul]	& A	\ar[u]	\ar[l]
  \end{tikzcd}
 \]
 and likewise for the rings of integral elements.
 At the level of adic spectra we obtain a projective system whose transition morphisms turn out to be open immersions.
 Hence we can glue them and one checks that
 \[
  Y \times_X Z = \bigcup_i \Spa(B \langle T \rangle_{M_{L,i}},B \langle T \rangle_{M_{L,i}}^+)/J_i.
 \]
 More details can be found in \cite{Hu96}, Proposition~1.2.2
\end{proof}

\begin{example}
 Let us compute the fiber product of the closed unit disc $\DD_k(0,1)$ with itself over~$k$.
 By Theorem~\ref{fiber_product_existence}~(ii) this is done by just taking the tensor product:
 \[
  (k \langle T \rangle,k^\circ \langle T \rangle) \widehat\otimes_{(k,k^\circ)} (k \langle S \rangle, k^\circ \langle T \rangle) = k \langle T,S \rangle,
 \]
 where the completion is taken for the $\varpi$-adic topology.
 Its adic spectrum is a polydisc (of dimension two):
 \[
  \DD_k(0,1) \times_{\Spa(k,k^\circ)} \DD_k(0,1) = \Spa(k \langle T,S \rangle,k^\circ \langle T,S \rangle).
 \]
\end{example}

\begin{example} \label{example_generic_fiber}
 In Section~\ref{formal_schemes_generic_fibers} we have computed the generic fiber of a formal scheme~$\cX$ over the valuation ring~$k^\circ$ of a nonarchimedean field~$k$.
 This can in fact be expressed as the fiber product of the two morphisms
 \[
  \begin{tikzcd}
							& \cX					\ar[d]	\\
   \Spa(k,k^\circ)	\ar[r]	& \Spa(k^\circ,k^\circ).
  \end{tikzcd}
 \]
 The horizontal morphism is of finite type.
 So by Theorem~\ref{fiber_product_existence}~(iii) the fiber product exists.
 If $X$ is locally of finite type over~$k^\circ$, it can even be computed in the easy way by Theorem~\ref{fiber_product_existence}~(ii) by just taking tensor products.
 For instance, this is the case for $\cX = \Spa(\Z_p \langle T \rangle)$, where we have seen in Example~\ref{generic_fiber_closed_disc} that
 \[
  \cX_\eta =\cX \times_{\Spa(\Z_p,\Z_p)} \Spa(\Q_p,\Z_p) = \DD_{\Q_p}(0,1)
 \]
 is the closed unit disc over~$\Q_p$.
 
 If~$\cX$ is only locally \emph{formally} of finite type over~$\Z_p$, we have to compute the fiber product by taking an ascending union of affinoids.
 This is in line with the example of $\cX = \Spa(\Z_p \llbracket T \rrbracket)$ that we have computed in Example~\ref{ZpT_analytic_locus} and Example~\ref{ZpT_generic_fiber}:
 \[
  \cX_\eta = \cX \times_{\Spa(\Z_p,\Z_p)} \Spa(\Q_p,\Z_p) = \bigcup_{n=1}^{\infty} R \left(\frac{T^n}{p}\right) = \Dcirc_{\Q_p}(0,1).
 \]
\end{example}

We can also compute the special fiber of the formal scheme $\cX/k^\circ$ with this formalism.
Remember that the topologically nilpotent elements~$k^{\circ \circ}$ form the maximal ideal of~$k^\circ$.
We then denote by
\[
 k^\succ := k^\circ/k^{\circ \circ}
\]
the corresponding residue field.
The projection $k^\circ \twoheadrightarrow k^\succ$ is a quotient map and defines a closed immersion
\[
 \Spa(k^\succ,k^\succ) \hookrightarrow \Spa(k^\circ,k^\circ).
\]
The special fiber of~$\cX$ is obtained as the fiber product
\[
 \cX_s := \cX \times_{\Spa(k^\circ,k^\circ)} \Spa(k^\succ,k^\succ).
\]

\begin{example}
 Let us compute the special fiber of $\cX =\Spa(\Z_p \langle T \rangle)$.
 Since all morphisms involved are adic, this is done by taking the tensor product
 \[
  \cX_s = \Spa(\Z_p \langle T \rangle \otimes_{\Z_p} \F_p) = \Spa(\F_p[T]).
 \]
 The topology on $\F_p[T]$ is the discrete topology, so $\cX_s$ is a discretely ringed adic space.
 In fact one can show (see \cite{Tem11}, \S~3.1) that for any field~$F$ (or in fact any base scheme) there is a fully faithful functor from the category of $F$-schemes to the category of discretely ringed adic spaces over $F$ mapping an affine scheme $\Spec A$ to $\Spa(A,F)$. 
 The special fiber~$\cX_s$ is of this kind for $F = \F_p$.
 So we can view it rather as the scheme $\Spec \F_p[T]$.
 
 Let us now take a look at $\cX = \Spa(\Z_p \llbracket T \rrbracket)$.
 According to the construction of the fiber product we first need to write~$\F_p$ as a quotient of some weighted Tate algebra over~$\Z_p$
 But $\F_p$ is a quotient of~$\Z_p$ on the nose without adding any variables.
 This is a degenerate situation where we do not have to add any variables to $\Z_p \llbracket T \rrbracket$ but just need to take the quotient by the ideal generated by~$p$.
 There are no affinoids to glue but we directly get
 \[
  \cX_s = \Spa(\Z_p \llbracket T \rrbracket/(p)) = \Spa(\F_p \llbracket T \rrbracket).
 \]
\end{example}

\section{Analytification} \label{section_analytification}

In complex geometry, starting from a variety $X$ over~$\CC$ we can equip $X(\CC)$ with the topology coming from~$\CC$ to obtain a topological space $X^\an$.
This is known as the analytification of~$X$.

For a variety~$X$ over~$\CC_p$ we want to do a similar construction.
Classically one could turn $X(\CC_p)$ into a rigid analytic variety.
However, from a more modern viewpoint one should treat the analytification as an adic space.
In this setting it looks more like a fiber product.

For an adic space $(X,\cO_X,\{v_x\}_{x \in X})$ we denote by
\[
 \underline{X} := (X,\cO_X)
\]
the underlying locally ringed space.
In other words, we forget the topology of the structure sheaf and we forget the valuations~$v_x$.

\begin{definition}
 Let $Y \to X$ be a morphism of schemes and $\underline{Z} \to X$ a morphism of locally ringed spaces where~$Z$ is an adic space.
 The \textit{fiber product} $Y \times_X Z$ is an adic space together with a morphism $Y \times_X Z \to Z$  whose underlying locally ringed space $\underline{Y \times_X Z}$ fits into a diagram
 \[
  \begin{tikzcd}
   \underline{Y \times_X Z}	\ar[d]	\ar[r]	& Y	\ar[d]	\\
   \underline{Z}					\ar[r]	& X.
  \end{tikzcd}
 \]
 satisfying the following universal property:
 For every adic space~$S$ and morphism $f:S \to Z$ and every solid arrow diagram
\[
 \begin{tikzcd}
  \underline{S}	\ar[drr,bend left]	\ar[dr,dashed,"\underline{g}"]	\ar[ddr,bend right, "\underline{f}"']	& [-25pt]	\\[-12pt]
																						& \underline{Y\times_X Z}		\ar[r]	\ar[d]	& Y	\ar[d]	\\
																						& \underline{Z}			\ar[r]			& X
 \end{tikzcd}
\]
there is a unique morphism of adic spaces $g:S \to Y \times_X Z$ such that the induced morphism~$\underline{g}$ (the dotted arrow) makes the diagram commutative.
\end{definition}

\begin{proposition} \label{analytification}
 Let $Y \to X$ be a morphism of schemes that is locally of finite type,~$Z$ a stable adic space, and $\underline{Z} \to X$ a morphism of locally ringed spaces.
 Then the fiber product $Y \times_X Z$ exists, is a stable adic space, and the projection
 \[
  Y \times_X Z \to Z
 \]
 is locally of finite type.
\end{proposition}

\begin{proof}
 We may assume~$X$ and~$Y$ are affine,
 \[
  X = \Spec A, \qquad	Y = \Spec A[T_1,\ldots,T_n]/I
 \]
 and~$Z$ is affinoid,
 \[
  Z = \Spa(B,B^+).
 \]
 with a stably sheafy Huber pair $(B,B^+)$.
 For simplicity we assume that~$B$ is a Tate ring with pseudouniformizer~$\varpi$.
 The general construction is not much more complicated and can be found in \cite{Wed19}, Proposition and Definition~8.61.
 For $k \in \NN$ we consider the Tate rings
 \[
  B\langle \varpi^k T_1,\ldots,\varpi^k T_n \rangle.
 \]
 Here $\varpi^k T_i$ are to be considered as variables.
 We chose to write them in this way to make the glueing process below more intuitive.
 In general one has to work with weighted Tate algebras and make the identification
 \[
  B\langle \varpi^k T_1,\ldots,\varpi^k T_n \rangle = B \langle T_1,\ldots,T_n \rangle_{(\varpi^k,\ldots,\varpi^k)}.
 \]
 We have associated rings of integral elements
 \[
  B^+\langle \varpi^k T_1,\ldots,\varpi^k T_n \rangle.
 \]
 For $m \ge k$ we have natural homomorphisms
 \[
  B\langle \varpi^m T_1,\ldots,\varpi^m T_n \rangle \longrightarrow B\langle \varpi^k T_1,\ldots,\varpi^k T_n \rangle
 \]
 sending $\varpi^m T_i$ to $\varpi^{m-k}(\varpi^k T_i)$ and accordingly for rings of integral elements.
 On adic spectra this corresponds to the embedding of the ball of radius $|\varpi^{-k}|$ into the ball of radius $|\varpi^{-m}|$.
 We have compatible embeddings
 \[
  B[T_1,\ldots,T_n] \hookrightarrow B\langle \varpi^k T_1,\ldots,\varpi^k T_n \rangle
 \]
 mapping $T_i$ to $\varpi^{-k}(\varpi^k T_i)$ and we view the ideal~$I$ as a subset of $B\langle \varpi^k T_1,\ldots,\varpi^k T_n \rangle$ via this embedding.
 Then
 \[
  B_k: = B\langle \varpi^k T_1,\ldots,\varpi^k T_n \rangle/IB\langle \varpi^k T_1,\ldots,\varpi^k T_n \rangle
 \]
 is a Huber ring topologically of finite type over~$B$.
 Defining $B_k^+$ to be the integral closure of $B^+\langle \varpi^k T_1,\ldots,\varpi^k T_n \rangle$ in~$B_k$ we obtain a Huber pair $(B_k,B_k^+)$ topologically of finite type over $(B,B^+)$.
 For $m \ge k$ we obtain compatible open immersions
 \[
  \Spa(B_k,B_k^+) \longrightarrow \Spa(B_m,B_m^+)
 \]
 that we can glue together.
 The resulting adic space turns out to be the fiber product:
 \[
  Y \times_X Z = \bigcup_k \Spa(B_k,B_k^+).
 \]
 In order to see this we need to check that for every map of sheafy Huber pairs $\varphi: (B,B^+) \to (C,C^+)$ and every commutative diagram
 \[
  \begin{tikzcd}
   C						& A[T_1,\ldots,T_n]/I	\ar[l]	\\
   B	\ar[u,"\varphi"]	& A						\ar[l]	\ar[u]
  \end{tikzcd}
 \]
 there is $k \in \N$ and a homomorphism $\psi:(B_k,B_k^+) \to (C,C^+)$ making the diagram
 \[
  \begin{tikzcd}
   C												& [-10pt]	\\ [-12pt]
													& B_k	\ar[ul,"\psi"']						& A[T_1,\ldots,T_n]/I	\ar[l]	\ar[ull,bend right]	\\
													& B	\ar[u]	\ar[uul,bend left,"\varphi"]	& A						\ar[l]	\ar[u]
  \end{tikzcd}
 \]
 commutative.
 Let $c_1,\ldots,c_n$ be the images of $T_1,\ldots,T_n$ in~$C$.
 For big enough~$k$ the elements $\varpi^k c_i$ are all contained in~$C^+$ as~$C^+$ is open and~$\varpi$ is topologically nilpotent.
 Thus we obtain a well defined homomorphism
 \[
  \psi: (B_k,B_k^+) \to (C,C^+)
 \]
 mapping $\varpi^k T_i$ to $\varpi^k c_i$ and it is clear from the construction that the corresponding diagram commutes.
\end{proof}

\begin{definition}
 Let $(k,k^+)$ be an affinoid field.
 The analytification of a variety $X/k$ is defined as
 \[
  X^\an := X \times_{\Spec k} \Spa(k,k^+).
 \]
\end{definition}

Let $(k,k^\circ)$ be a nonarchimedean field.
We now have two ways of constructing an adic space over~$k$ out of a scheme~$X$ of finite type over~$k$.
The mechanism discussed in this section is the analytification resulting in~$X^\an$.
But we could also choose a model of~$X$ over~$k^\circ$, complete it with respect to a pseudouniformizer to obtain a formal scheme~$\cX$, and then take the generic fiber.
However, these two constructions do not lead to the same result as the following example shows.

\begin{example}
 Let us consider the affine line~$\A_k^1$ (as a scheme) over the nonarchimedean field $(k,k^\circ)$.
 We first compute its analytification.
 According to the construction in the proof of Proposition~\ref{analytification} we just need to glue the affinoids
 \[
  \Spa(k \langle \varpi^k T \rangle)
 \]
 for $k \in \N$.
 But we have seen this space in Example~\ref{example_affine_line}, it is the adic affine line:
 \[
  (\A_k^1)^{\ad} = \A_{(k,k^\circ)}^1.
 \]
 This closes the circle and justifies the nomenclature.
 
 Let us now go the other way via generic fibers of models.
 An obvious model of $\A_k^1$ over~$k^\circ$ is $\A_{k^\circ}^1= \Spec k^\circ[T]$.
 Completing it along the special fiber corresponds to taking the $\varpi$-adic completion of $k^\circ[T]$, i.e. the ring of convergent power series $k^\circ \langle T \rangle$.
 Consequently we have to compute the generic fiber of the formal scheme
 \[
  \Spa(k^\circ \langle T \rangle).
 \]
 We have done this in  Example~\ref{generic_fiber_closed_disc}.
 The resulting adic space is the closed unit disc $\DD_k(0,1)$.
\end{example}

In analogy to the analytification of complex varieties, nonarchimedean analytification satisfies a GAGA theorem.
It relates coherent sheaves on a proper variety~$X$ with coherent sheaves on its analytification~$X^\an$.
In order to treat this GAGA theorem we first have to define coherent sheaves on adic spaces, which will be done in Christian Johansson's lecture (\cite{JohTopicsAdic}) in section~1.1.
The GAGA theorem is also stated there as Theorem~1.2.19.


\subsection*{Acknowledgments}
 First of all I would like to thank the organizers of the spring school on non-archimedean geometry and eigenvarieties that took place in Heidelberg in spring 2023.
 Without their invitation to me to give a course and their managment of the editorial process to publish proceedings these notes would not have come into extistence.
 Moreover, my thanks go to Max Witzelsperger whose concept for the exercise session to this course also contributed to these notes.
 I am also grateful for his throrough proofreading of the manuscript.

\subsection*{Funding}
This project was funded by the Deutsche Forschungsgemeinschaft
(DFG, German Research Foundation), TRR 326 \emph{Geometry and Arithmetic of Uniformized Structures}, project number 444845124.

\bibliographystyle{meinStil}
\bibliography{citations}

\end{document}